\documentclass[final]{siamart1116}

\usepackage{braket,amsfonts,enumerate,mathtools}

\usepackage{array}

\usepackage[caption=false]{subfig}
\usepackage{pgfplots}

\newsiamthm{claim}{Claim}
\newsiamthm{assumption}{Assumption}
\newsiamremark{rem}{Remark}
\newsiamremark{expl}{Example}
\newsiamremark{hypothesis}{Hypothesis}
\crefname{hypothesis}{Hypothesis}{Hypotheses}

\usepackage{algorithmic}

\usepackage{graphicx,epstopdf}

\Crefname{ALC@unique}{Line}{Lines}

\usepackage{amsopn}

\numberwithin{theorem}{section}

\usepackage{xspace}
\usepackage{bold-extra}
\usepackage[most]{tcolorbox}

\colorlet{texcscolor}{blue!50!black}
\colorlet{texemcolor}{red!70!black}
\colorlet{texpreamble}{red!70!black}
\colorlet{codebackground}{black!25!white!25}

\lstdefinestyle{siamlatex}{%
  style=tcblatex,
  texcsstyle=*\color{texcscolor},
  texcsstyle=[2]\color{texemcolor},
  keywordstyle=[2]\color{texemcolor},
  moretexcs={cref,Cref,maketitle,mathcal,text,headers,email,url},
}

\tcbset{%
  colframe=black!75!white!75,
  coltitle=white,
  colback=codebackground, 
  colbacklower=white, 
  fonttitle=\bfseries,
  arc=0pt,outer arc=0pt,
  top=1pt,bottom=1pt,left=1mm,right=1mm,middle=1mm,boxsep=1mm,
  leftrule=0.3mm,rightrule=0.3mm,toprule=0.3mm,bottomrule=0.3mm,
  listing options={style=siamlatex}
}

\newtcblisting[use counter=example]{example}[2][]{%
  title={Example~\thetcbcounter: #2},#1}

\newtcbinputlisting[use counter=example]{\examplefile}[3][]{%
  title={Example~\thetcbcounter: #2},listing file={#3},#1}

\DeclareTotalTCBox{\code}{ v O{} }
{ 
  fontupper=\ttfamily\color{black},
  nobeforeafter,
  tcbox raise base,
  colback=codebackground,colframe=white,
  top=0pt,bottom=0pt,left=0mm,right=0mm,
  leftrule=0pt,rightrule=0pt,toprule=0mm,bottomrule=0mm,
  boxsep=0.5mm,
  #2}{#1}

\patchcmd\newpage{\vfil}{}{}{}
\begin{tcbverbatimwrite}{tmp_\jobname_header.tex}
\title{Anomalous Diffusion and the Generalized Langevin Equation%
  \thanks{Submitted November 1, 2017.
\funding{This work has been supported by DMS-1644290}}}

\author{Scott A McKinley%
  \thanks{Department of Mathematics, Tulane University, New Orleans, LA (\email{scott.mckinley@tulane.edu}).}%
  \and
  Hung D Nguyen%
  \thanks{Department of Mathematics, Tulane University, New Orleans, LA. (\email{hnguye25@tulane.edu}).}
}

\headers{Anomalous Diffusion and the Generalized Langevin Equation}
{Scott A McKinley and Hung D Nguyen}
\end{tcbverbatimwrite}
\input{tmp_\jobname_header.tex}

\ifpdf
\hypersetup{ pdftitle={Anomalous Diffusion and the Generalized Langevin Equation} }
\fi

\newcommand{\nbb}{\mathbb{N}}
\newcommand{\rbb}{\mathbb{R}}
\newcommand{\zbb}{\mathbb{Z}}
\newcommand{\cbb}{\mathbb{C}}

\newcommand{\Sc}{\mathcal{S}}
\newcommand{\F}[1]{\mathcal{F}\left[#1\right]}
\newcommand{\Fcos}{\mathcal{F}_{\cos}}
\newcommand{\Fsin}{\mathcal{F}_{\sin}}
\newcommand{\Kcos}{\mathcal{K}_{\cos}}
\newcommand{\Ksin}{\mathcal{K}_{\sin}}
\newcommand{\rhat}{\widehat{r}}

\newcommand{\la}{\langle}
\newcommand{\ra}{\rangle}

\renewcommand{\d}{\mathrm{d}}

\newcommand{\f}{\varphi}

\newcommand{\suchthat}{\, : \,}
\newcommand{\E}[1]{\mathbb{E}\left[#1\right]}
\newcommand{\Enone}[1]{\mathbb{E}#1}

\begin{document}

\maketitle

\begin{abstract}
The Generalized Langevin Equation (GLE) is a Stochastic Integro-Differential Equation that is commonly used to describe the velocity of microparticles that move randomly in viscoelastic fluids.  Such particles commonly exhibit what is known as anomalous subdiffusion, which is to say that their position Mean-Squared Displacement (MSD) scales sublinearly with time. While it is common in the literature to observe that there is a relationship between the MSD and the memory structure of the GLE, and there exist special cases where explicit solutions exist, this connection has never been fully characterized. Here, we establish a class of memory kernels for which the GLE is well-defined; we investigate the associated regularity properties of solutions; and we prove that large-time asymptotic behavior of the particle MSD is entirely determined by the tail behavior of the GLE's memory kernel. 
%
\end{abstract}

\section{Introduction}

The Generalized Langevin Equation (GLE) is a Stochastic Integro-Differential Equation that is now a commonly used to describe the velocity of micro-particles diffusing in viscoelastic fluids. Introduced by Mori in 1965 \cite{mori1965transport} and Kubo in 1966 \cite{kubo1966fluctuation}, then popularized for modeling viscoelastic diffusion by Mason \& Weitz in 1995 \cite{mason1995optical}, the GLE is a balance-of-forces equation that features a prominent memory effect. Let $\{X(t)\}_{t \geq 0}$ and $\{V(t)\}_{t \in \rbb}$ be stochastic processes denoting a particle's time-dependent position and velocity. For the sake of simplicity, we will consider these processes to be one-dimensional, but this has no impact on our major findings. There are several perspectives on how the GLE can be derived from heat bath models \cite{kupferman2004fractional,kou2008stochastic} or from principles of polymer physics \cite{fricks2009time} and viscoelastic fluid theory \cite{goychuk2012viscoelastic,indei2012treating}. With slight notational changes, we consider the version of the GLE that appears in \cite{hohenegger2017fluid}, which has the most general form:
\begin{equation} \label{eq:gle}
m \d V(t)	= -\lambda V(t) - \beta \int_{-\infty}^t \!\!\!\! K(t-s) V(s) \d s + \sqrt{\beta} F(t) \d t + \sqrt{2 \lambda} \d W(t),
\end{equation}
where $m$ is the particle's mass, $\lambda$ and $\beta$ represent the particle's viscous and elastic drag coefficients, and $K: \rbb \to \rbb_+$ is a memory kernel that summarizes how the surrounding fluid stores kinetic energy from the particle and then acts back on the particle at a later time. The process $\{W(t)\}_{t \in \rbb}$ is a two-sided standard Brownian motion, while $\{F(t)\}_{t \in \rbb}$ is a mean zero, stationary, Gaussian process with covariance
\begin{equation} \label{eq:gle-F}
	\E{F(t) F(s)} = K(t-s).
\end{equation}
The fact that we require the covariance of $F(t)$ to be the same function as the memory kernel appearing in \eqref{eq:gle} is a manifestation of the Fluctuation-Dissipation relationship \cite{kubo1966fluctuation}. To have correct physical units, the coefficients of $F(t)$ and $\d W(t)$ should be $\sqrt{\beta k_B T}$ and $\sqrt{2 \lambda k_B T}$, respectively, where $k_B$ is Boltzman's constant and $T$ is the temperature of the system, but we will ignore this factor throughout this work. The reason why there is a two in the coefficient of $\d W(t)$ but not $F(t)$ is in order to satisfy equipartition of energy, as discussed in \cite{hohenegger2017equipartition} and \cite{hohenegger2017fluid}.

The GLE is one of a few qualitatively distinct mathematical models that can produce what is known as \emph{anomalous diffusion}. A particle position process $X(t) := \int_0^t V(s) \d s$, $t \geq 0$ (sometimes referred to as the integrated GLE (iGLE)) is said to be \emph{diffusive} if its Mean-Squared Displacement (MSD), $\E{X^2(t)}$, satisfies $\E{X^2(t)} = C t$ for some constant $C > 0$ for all time $t$. Any departure from being diffusive qualifies a process as exhibiting \emph{anomalous diffusion}. Single particle tracking experiments for a wide variety of particles in biological fluids feature particles that exhibit anomalous \emph{subdiffusion}, which is to say that for a large segment of time $\E{|X|^2(t)} \approx C t^{\alpha}$ for some $\alpha \in (0,1)$ \cite{golding2006physical,ernst2012fractional,hill2014biophysical}.

We will mostly concern ourselves with large-$t$ behavior and whether an iGLE has the following property:
\begin{equation}
	\text{\emph{Asymptotically Subdiffusive} } X(t): \,\,\, \E{X^2(t)} \sim t^{\alpha} \text{ as } t \to \infty,\hspace*{2 cm}
\end{equation}
where, for two functions $f$ and $g$, we say 
\begin{displaymath}
	f(t) \sim g(t) \text{ as } t \to \infty \text { if, for some } C \in (0,\infty), \, \lim_{t \to \infty} f(t)/g(t) = C.
\end{displaymath}

The large-time MSD behavior of the iGLE is entirely determined by its memory kernel $K(t)$. To our knowledge, Morgado et al.~(2002) \cite{morgado2002relation} were the first to make this relationship explicit:
\begin{equation} \label{eq:meta-thm}
	\text{\emph{Meta-Theorem}: } \text{for } \alpha \in (0,1), \,\,  K(t) \sim t^{-\alpha} \implies \E{X^2(t)} \sim t^{\alpha}, \text{ as } t \to \infty.
\end{equation} 

The argument presented by Morgado et al.~was informal and Kneller (2011) \cite{kneller2011generalized} later presented an attempt to make it rigorous. Both arguments rely a chain of three relationships:
\begin{enumerate}[(i)]
\item relating the MSD to the Autocovariance Function (ACF, $r(t) := \E{V(t)V(s)}$);
\item relating the Laplace transform of the ACF to the  Laplace transform of $K$;
\item relating the Laplace transform of $K$ near zero to $K(t)$ itself for large $t$.
\end{enumerate}

Relationship (i) follows from the classical formula \cite{reif1965fundamentals}:
\begin{displaymath}
	\E{X^2(t)} = 2 \int_0^t (t-s) r(s) \d s.
\end{displaymath}
Relationship (iii) follows from the Hardy-Littlewood-Karamata (HLK) Tauberian Theorem for Laplace transforms \cite{feller1966introduction}. However, it has recently been shown that the proposed Relationship (ii) is not valid \cite{hohenegger2017reconstructing}. The reason is that these arguments rely on a widely cited assumption that $E{F(t) V(0)} = 0$ for all $ t > 0$ in stationarity. This is, in fact, not the case for stationary solutions to the GLE. However, the assumption appears, for example, the seminal works by Kubo (1966) \cite{kubo1966fluctuation}, Mason (2000) \cite{mason2000estimating} and Squires \& Mason (2010) \cite{squires2010fluid}).

There are some special cases in which rigorous work has been done on the Meta-Theorem. In 2004, Kupferman \cite{kupferman2004fractional} studied a version of the GLE where $\lambda = 0$ and the convolution integral in \eqref{eq:gle} is defined on the interval $[0,t]$ rather than $(-\infty, t]$. In this system, by assuming $K(t) = C t^{-\alpha}$, the author derived an exact solution and demonstrated that the MSD scales like $t^\alpha$. In 2008, Kou presented the GLE ($\lambda = 0$) defined with the convolution over $(-\infty, t]$ and $K(t) = C t^{-\alpha}$. Importantly, Kou shifted the analysis to a Fourier transform setting (more natural for studying a stationary process like $V$) and proved the Meta-Theorem holds in this special case. Later, in 2012, Didier et al.~\cite{didier2012statistical} introduced a condition on the Fourier transform of the spectral density of solutions (as the frequency tends to 0) that predicts the large-time scaling of the MSD. The limit theorem takes the form that there exist positive constants $c$ and $C$ such that $c \leq \lim_{t \to \infty} \E{X^2(t)}/t^\alpha \leq C$. However, the stated condition is not easy to interpret as a condition directly on $K(t)$.

In the work that follows we establish a large class of memory kernels $K(t)$ for which the GLE and iGLE are well-posed. We analyze regularity of the solutions and are able to characterize the large-$t$ asymptotics of the MSD of $X(t)$ as follows: if $K(t)$ is integrable, then $X(t)$ is asymptotically diffusive; if $K(t)$ is not integrable, but has nice behavior for large $t$, then the Meta-Theorem \eqref{eq:meta-thm} holds. In Section \ref{sec:intro:K} we lay out sufficient assumptions for $K(t)$ in two cases -- in the first case (Assumption \ref{a:K}), when either $m > 0$ or $\lambda > 0$, and in the second case (Assumption \ref{cond:cv}), when $m = \lambda = 0$. Moreover, we describe some important memory kernel examples in the literature.  In Section \ref{sec:intro:summary} we provide a rigorous summary of our results including our version of the Meta-Theorem \eqref{eq:meta-thm}, namely Theorem \ref{thm:subdiffusion}.

\subsection{The class of admissible memory functions $K(t)$}
\label{sec:intro:K}


The two primary examples of memory kernels from the literature are:
\begin{equation}
	\begin{aligned}
		\text{Sum of exponentials:}& \quad K(t) = \sum_{k=1}^n c_k e^{-\lambda_k} \,\,\, \cite{fricks2009time, goychuk2012viscoelastic, pavliotis2014stochastic, lysy2016model,hall2016uncertainty}; \text{ and} \\
		\text{Power law:}& \quad K(t) = c_\alpha t^{-\alpha}, (\alpha \in (0,1)) \, \quad \,\cite{kupferman2004fractional,kou2008stochastic}.
	\end{aligned}
\end{equation}
The coefficients of the sum of exponentials $\{c_k\}$ are real numbers and the coefficient $c_\alpha$ is an $\alpha$-dependent positive constant. We generalize these examples as follows.

\begin{assumption} \label{a:K}
Given $K:\rbb\to\rbb$, where $K(0)$ may be infinite, we assume:
\begin{enumerate}[(I)]
\item \label{cond1} \begin{enumerate}[a.]
\item $K$ is symmetric and positive for all  non-zero $t$;
\item \label{cond1b} $K(t) \to 0$ as $t \to \infty$ and is eventually decreasing;
\item \label{cond1c}$K \in L^1_{loc}(\rbb)$;
\item \label{cond1d}The improper integral $\Kcos(\omega)=\int_0^\infty \! K(t)\cos(\omega t) \, \d t$ is positive for all non-zero $\omega$.
\end{enumerate}
\end{enumerate}
Furthermore, either
\begin{enumerate}[(I)]
\addtocounter{enumi}{1}
\item $K \in L^1$,\label{cond2a}
\end{enumerate}
or
\begin{enumerate}[(I)]
\addtocounter{enumi}{2}
\item $K \notin L^1$ but there exists $\alpha\in(0,1)$ such that $K(t) \sim t^{-\alpha}$ as $t \to \infty$.\label{cond2b}
\end{enumerate}	
\end{assumption}

Assumption \ref{a:K} is sufficient as long as either $m > 0$ or $\lambda > 0$. If $m = \lambda = 0$, then we need to introduce stricter conditions. Most notably, $K$ will need to be convex.

\begin{assumption}[Extension when $\lambda = \mu = 0$]
\label{cond:cv} 
Given $K:\rbb\to\rbb$ where $K(0)$ may be infinite, we assume: 
\begin{enumerate}[(I)]
\addtocounter{enumi}{3}
\item \label{cond:convex} $K\in C^2(0,\infty)$ is convex and $K''(t)$ is monotone near the origin.
\end{enumerate} 
Furthermore, either
	\begin{enumerate}[(I)]
		\addtocounter{enumi}{4}
		\item \label{cond:cv-finite} $K(0)$ is finite and there exists $\sigma_1\in(0,1)$ such that $\lim_{t\to 0}t^{\sigma_1} K'(t)=0$.
	\end{enumerate}
or
	\begin{enumerate}[(I)]
	\addtocounter{enumi}{5}
		\item \label{cond:cv-infinite} $K(0)$ is infinite but there exists $\sigma_2\in(0,1)$ such that $\lim_{t\to 0}t^{\sigma_2} K(t)\in(0,\infty)$. 
	\end{enumerate}
\end{assumption}

It has been noted in many places (recently in \cite{mckinley2009transient, goychuk2012viscoelastic}) that a sum of exponentials with sufficiently many terms can be used to approximate functions that have power-law behavior for large-$t$, but diverse behavior near the origin. As we note in Section \ref{sec:prelim:exponentials}, the ``closure'' of the family of sum-of-exponential functions, namely the \emph{completely monotone functions}, satisfy the conditions of Assumption \ref{a:K}.

\subsection{Summary of Results}
\label{sec:intro:summary}

In Section \ref{sec:prelim}, we lay out the mathematical foundation on which our main theorems are built. The results in Sections \ref{sec:prelim:tempered}-\ref{sec:prelim:stationary} are a review of necessary definitions, notation, and results from classical stationary process theory (with a modest extension in Section \ref{sec:prelim:extension}). Much of the work in Section \ref{sec:tauberian} is inspired from previous work by Soni and Soni (1975), \cite{soni1975slowly}. Namely, we prove some Abelian theorems for improper Fourier transforms that are necessary for our asymptotic analysis of the MSD in the subdiffusive case.

In Section \ref{sec:weak-solutions} we establish our notion of weak solutions for GLE/iGLE pairs, and in Section \ref{sec:regularity} we provide conditions on $K(t)$ and the parameters $m$ and $\lambda$ that lead to continuous (or differentiable) versions of $V(t)$. The parameters $m$ and $\lambda$ play a prominent role here, and it does not matter whether the process is asymptotically diffusive or subdiffusive. We summarize these results as follows.

Suppose that $K(t)$ satisfies Condition \ref{cond1}. Then if $m > 0$ or $\lambda > 0$, the GLE is well-posed and we find the following:
\begin{equation}
\begin{aligned}
	m > 0, \, \lambda > 0:& \quad V(t) \text{ is continuous a.s.} \\
	m > 0, \, \lambda = 0:& \quad V(t) \text{ is continuous a.s.}^\dagger \\
	m = 0, \, \lambda > 0:& \quad X(t) \text{ is continuous a.s.}
\end{aligned}	
\end{equation}
In the last case, we understand the velocity process $V$ in the sense of stationary random distributions. The $\dagger$ indicates that, in the ($m > 0, \, \lambda = 0$) case, stricter conditions can be placed on $K(t)$ so that $V(t)$ is, in fact, differentiable (see Theorem \ref{thm:V-differentiable}). 

To address the $m = \lambda = 0$ case, we must impose further conditions. Namely, suppose that, in addition to $\ref{cond1}$, $K(t)$ satisfies Condition \ref{cond:convex} and either \ref{cond:cv-finite} or \ref{cond:cv-infinite}. Then the GLE is well-posed and
\begin{equation}
	m = 0, \, \lambda = 0: \quad X(t) \text{ is continuous a.s.}
\end{equation}
Again, we understand $V$ in the sense of stationary random distributions.

With these regularity results in hand, we proceed in Section \ref{sec:msd} to prove our \textbf{main theorem} on the dichotomy between being asymptotically diffusive or subdiffusive, Theorem \ref{thm:subdiffusion}.

\begin{theorem}[Asymptotic Behavior of the MSD] \label{thm:subdiffusion}
Let $\{V(t)\}_{t \in \rbb}$ be a solution to the GLE in the sense defined in Definition \ref{def:weak-solution} and let $\{X(t)\}_{t \geq 0}$ be the associated iGLE. If $m > 0$ or $\lambda > 0$, then
\begin{equation}
	\lim_{t \to \infty} \E{X^2(t)} \sim t^{\eta}, \text{ where } \eta = \left\{
	\begin{array}{rl}
	1, & \text{if } K(t) \text{ satisfies }	 \eqref{cond1}+\eqref{cond2a} \\
	\alpha, & \text{if } K(t) \text{ satisfies }	 \eqref{cond1}+\eqref{cond2b}	
	\end{array}\right.
\end{equation}
where, in the latter case, $\alpha \in (0,1)$ is the constant from Assumption \eqref{cond2b}. 

If $m = \lambda = 0$, then the Condition \eqref{cond1} should be replaced with \eqref{cond1} $+$ Assumption \ref{cond:cv}.
\end{theorem}
This is our version of the our version of the Meta-Theorem \eqref{eq:meta-thm} and the proof appears in Section \ref{sec:msd}.

Finally, as has been noted in several places \cite{mckinley2009transient,goychuk2012viscoelastic,pavliotis2014stochastic}, a process might be asymptotically diffusive, but nevertheless exhibit anomalous behavior over a very large time range. In Section \ref{sec:transient-anomalous-diffusion}, we provide a rigorous definition for so-called \emph{transient anomalous diffusion} and characterize one important setting in which it arises.

\section{Mathematical Preliminaries} 
\label{sec:prelim}

\subsection{Tempered Distributions and Fourier Transform}
\label{sec:prelim:tempered}
For a function $f:\rbb\to\cbb$, we define the Fourier transform of $f$ and its inverse as
\begin{displaymath}
\widehat{f}(\omega)=\int_\rbb f(t) e^{-it\omega} \d t, \text{ and }	 \check{f}(t)=\frac{1}{2\pi}\int_\rbb f(\omega) e^{it\omega}\d \omega.
\end{displaymath}

We use $\Sc$ to denote the class of Schwarz functions and $\Sc'$ for the class of tempered distributions on $\Sc$. 
For $g\in\Sc'$, we write $\F{g}$ for the Fourier transform of $g$ in $\Sc'$. That is to say, for all $\varphi\in\Sc$, it holds that
\begin{displaymath}
\langle g,\widehat{\varphi}\rangle = \langle \F{g},\varphi\rangle.	
\end{displaymath}

\subsection{Positive Definiteness}
\label{sec:prelim:positive-definiteness}

We recall some basic definitions and theorems that can be found, for example, in the text by Cram\'{e}r and Leadbetter \cite{cramer1967stationary}.
\begin{definition}
A continuous function $r:\rbb\to\cbb$ is positive definite if the following holds
\[\sum_{j,k=1}^n r(t_j-t_k)z_j\overline{z}_k\geq 0,\]
for any finite set of time points $t_j$ and complex numbers $z_j$.	
\end{definition}

\begin{theorem}[Bochner's Theorem]\label{thm:Bochner}  A function $f(t)$ is positive definite if and only if it can be represented in the form
\begin{displaymath}
f(t)=\int_\rbb e^{it\omega}\nu(\d\omega),	
\end{displaymath}
where $\nu$ is a positive finite Borel measure.
\end{theorem}

When the measure $\nu$ has a density $\widehat{f}$, i.e.\ the covariance $f$ admits the formula $f(t)=\int_\rbb e^{it\omega}\widehat{f}(\omega)\d\omega$, then $\widehat{f}$ is called the spectral density. In fact, this is guaranteed by the first condition we impose on our memory kernels. 

\begin{proposition} \label{prop:Fourier-Inverse} 
	Let $f$ be a positive definite function satisfying \eqref{cond1b}. Then, $f$ admits the inverse Fourier formula
\begin{equation}\label{eqn:Fourier-Inverse}
f(t) = \frac{1}{\pi}\int_\rbb \widehat{f}(\omega)e^{it\omega}\d\omega,
\end{equation}
where $\widehat{f}(\omega) = \int_0^\infty f(t)\cos(t\omega)\d t$.
\end{proposition}

The proof of Proposition \ref{prop:Fourier-Inverse} can be found in \cite{inoue1995abel}, Theorem 5.1. The inversion formula~\eqref{eqn:Fourier-Inverse} will be useful in Section \ref{sec:regularity:mpos-lambdazero} where we investigate the differentiability of solutions to the GLE.

In order to make sense of the GLE in general, we will need the theory of stationary random distributions, introduced by It\^{o} \cite{ito1954stationary}. This requires an extension of the notion of positive definiteness to the tempered distributions.

\begin{definition} A tempered distribution $f\in\Sc'$ is called {\it positive definite} if for any $\varphi\in \Sc$,
\[ \langle r,\varphi*\widetilde{\varphi}\rangle \geq 0,  \]
where $\widetilde{\varphi}(x)=\varphi(-x)$.
\end{definition}

Much as Bochner's Theorem characterizes the positive definite functions, there is a characterization of positive definite tempered distributions as well.

\begin{theorem}[\cite{ito1954stationary}]\label{thm:pd-td}
A tempered distribution $f$ is positive definite if and only if $f$ admits a representation 
\begin{displaymath}
	\langle f,\varphi\rangle = \int_\rbb \overline{\widehat{\varphi}(\omega)}\nu(\d\omega),	
\end{displaymath}
where $\nu$ is a non-negative measure on $\rbb$ satisfying
\begin{equation}\label{ineq:spectral-measure}
	\int_\rbb \frac{\nu(\d x)}{(1+x^2)^k}<\infty,
\end{equation}
for some integer $k$.
\end{theorem}
\begin{rem} Analogous to Theorem \ref{thm:Bochner}, when the measure $\nu$ in Theorem \ref{thm:pd-td} is absolutely continuous to Lebesgue measure (i.e.~if there exists a function $\widehat{f}$ such that $\nu(\d \omega)=\widehat{f}(\omega)\d\omega$), then $\widehat{f}$ is called the {\it spectral density} of the tempered distribution $f$. 
\end{rem}

\subsection{Stationary Random Processes and Stationary Random Distributions}
\label{sec:prelim:stationary}

\begin{definition}
 A stochastic process $\{F(t)\}_{t \in \rbb}$ is mean-square continuous and stationary if for all $t, s\in\rbb$,
		\begin{enumerate}[(a)]
		\item $\E{\lvert F(t) \rvert^2}<\infty$ and $\lim_{h\to 0}\E{\lvert F(t+h)-F(t) \rvert^2} =0$;
		\item $\E{F(t)}=a$, for some constant $a$ (we may assume $a=0$); and
		\item the covariance function $\E{F(t)\overline{F(s)}}$ only depends on the difference $(t-s)$.
		\end{enumerate}
\end{definition}
This definition of stationarity is often called \emph{stationary in the wide sense} but we will simply call such processes \emph{stationary}. The following connection between positive definite functions and covariance functions is explained, for example, in \cite{cramer1967stationary}.

\begin{theorem}\label{thm1} A function $r(t)$ is positive definite if and only if it is the covariance function of some mean-square continuous stationary process V(t), i.e.
\[r(t-s)=\E{F(t)\overline{F(s)}}.\]
$V$ can be chosen to be Gaussian.
\end{theorem}

The generalization of a stationary random \emph{process} is a stationary random \emph{distribution}, an idea introduced by It\^{o} in 1954 \cite{ito1954stationary}. Denote by $\tau_h$, the shift transform on $\Sc$, $\tau_h\varphi(x):=\varphi(x+h)$.
\begin{definition} A linear functional $F:\Sc\to L^2(\Omega)$, the space of all random variables with finite variance, is called a {\it stationary random distribution} on $\Sc$ if for all $h\in\rbb$, $\varphi_1,\varphi_2\in\Sc$,
\[\E{\langle  F,\tau_h \varphi_1\rangle \overline{\langle F,\tau_h\varphi_2  \rangle}}=\E{\langle  F,\varphi_1\rangle \overline{\langle F,\varphi_2  \rangle}}.\]
\end{definition}

\begin{definition} A process $\{\xi(t)\}_{t \in \rbb}$ is said to have {\it orthogonal increments} if, for any $t_1<t_2\leq t_3<t_4$, we have
\[\E{\left(\xi\left(t_4\right)-\xi\left(t_3\right)\right)\overline{\left(\xi\left(t_2\right)-\xi\left(t_1\right)\right)}}=0.\]
\end{definition}

\begin{theorem}[\cite{cramer1967stationary}] \label{thm:spectral} A process $\{F(t)\}_{t \in \rbb}$ is stationary if and only if there exists a stochastic process $\{\xi(\omega)\}_{\omega \in \rbb}$ with orthogonal increments such that for every $t\in\rbb$,
\[F(t)=\int_\rbb e^{it\omega}\xi(\d\omega).\]
\end{theorem}

\begin{theorem}[Characterization of Stationary Random Distributions \cite{ito1954stationary}] \label{thm:Ito-1}
 A tempered distribution $r$ is positive definite if and only if there exists a stationary random distribution $F$ such that for all $\varphi_1,\varphi_2\in\Sc$,
\[\E{\langle  F,\varphi_1\rangle \overline{\langle F,\varphi_2  \rangle}}=\langle r,\varphi_1*\widetilde{\varphi_2}\rangle.\]
$r$ is called the covariance distribution of $F$.
\end{theorem}
Recalling Theorem \ref{thm:pd-td}, $r$ can be represented by a non-negative measure $\nu$. We call $\nu$ the \emph{spectral measure} of $F$.

Next, we recall definition of {\it random measure}.
\begin{definition}[\cite{ito1954stationary}] \label{defn:random-measure}  Let $\mu$ be a non-negative measure on $\rbb$. Denote by $\mathcal{B}_\mu$, the collection of all Borel sets $E$ such that $\mu(E)<\infty$. A map $\xi:\mathcal{B}_\mu\to L^2(\Omega)$ is called a {\it random measure} with respect to $\mu$ if for $E_1, E_2\in \mathcal{B}_\mu$,
\[\E{\xi(E_1)\overline{\xi(E_2)}}=\mu(E_1\cap E_2).\]
\end{definition}

\begin{theorem} \label{thm:Ito-2}
Let $\{F(\f)\}_{\f \in \Sc}$ be a stationary random distribution with spectral measure $\nu$. Then, there exists a random measure $\xi$ that is defined with respect to $\nu$ such that
\begin{displaymath}
\langle F,\varphi\rangle = \int_\rbb \overline{\widehat{\varphi}(\omega) } \xi(\d\omega).	
\end{displaymath}
\end{theorem}

\subsection{An extension of the stationary random distributions} \label{sec:prelim:extension}

Let $\nu$ be the non-negative measure on $\rbb$ satisfying \eqref{ineq:spectral-measure} for some $k \in \zbb$. Denote by $L^2(\nu)$ the Hilbert space of equivalence classes of non-random complex-valued functions $g$ such that $\int_\rbb \left|g(s)\right|^2\nu( \d s)<\infty.$
Let $\xi$ be a random measure with respect to $\nu$ as in Definition \ref{defn:random-measure}. For every $g\in L^2(\nu)$, the stochastic integral $\int_\rbb g(s)\xi(\d s)$ is a well-defined mean zero Gaussian random variable with \[ \E{\int_\rbb g_1(s)\xi(\d s)\overline{\int_\rbb g_2(s)\xi(\d s)}}=\int_\rbb g_1(s)\overline{g_2(s)}\nu(\d s).\]
See \cite{ito1954stationary} for a detailed discussion.

As detailed above, there is a stationary random distribution $F \suchthat \Sc \to L^2(\Omega)$ whose spectral measure is $\nu$. If, we additionally have that $\nu$ is absolutely continuous to Lebesgue measure, we may extend $F$ to be an operator on $\Sc'$ as follows: for $g \in \Sc'$, let $\Phi \, : \, \Sc' \to L^2(\Omega)$ be defined as
\begin{equation} \label{eqn:V.op}
\langle \Phi,g\rangle = \int_\rbb \overline{\F{g}(\omega)}\xi(\d\omega).
\end{equation}
The domain of $\Phi$, denoted by $\text{Dom}(\Phi)$, is the set of tempered distributions $g$ such that its Fourier transform $\F{g}$ in $\Sc'$ is a function defined on $\rbb$ and that $\F{g}\in L^2(\nu)$. We stress that absolute continuity of $\nu$ with respect to Lebesgue measure is required in order to guarantee that the extension of $F$ is well-defined. To be precise, we have the following Lemma.
\begin{lemma} \label{lem:V-op-defined} Let $F:\Sc\to L^2(\Omega)$ be a stationary random distribution with spectral measure $nu$ and associated random measure $\xi$. Let $\Phi: \Sc'\to L^2(\Omega)$ be the extension of $F$ defined as by \eqref{eqn:V.op}. Assume further that $\nu$ is absolutely continuous to Lebesgue measure. Then, $\Phi$ is well-defined.
\end{lemma}
\begin{proof} Since $\nu$ is absolutely continuous with respect to Lebesgue measure, $\nu(\d\omega)=\widehat{r}(\omega)\d\omega$ for some function $\widehat{r}$. It suffices to show that the RHS of \eqref{eqn:V.op} does not depend on the choice of $\F{g}$. To see that, suppose $\mathcal{F}_1[g]$ and $\mathcal{F}_2[g]$ are Fourier transforms of $g$ in $\Sc'$, it is known that they must agree almost everywhere. We then have a chain of implication
\begin{equation}\label{eqn:V-op-defined-1}
\begin{aligned}
&\Enone{\Big|\int_\rbb \overline{\mathcal{F}_1[g](\omega)}\xi(\d\omega)-\int_\rbb \overline{\mathcal{F}_2[g](\omega)}\xi(\d\omega)\Big|^2} \\
&\qquad = \int_\rbb \left|\mathcal{F}_1[g](\omega)-\mathcal{F}_1[g](\omega)\right|^2\nu(\d\omega)=\int_\rbb \left|\mathcal{F}_1[g](\omega)-\mathcal{F}_1[g](\omega)\right|^2\widehat{r}(\d\omega)=0. 
\end{aligned}
\end{equation}
It follows that two random variables $\int_\rbb \overline{\mathcal{F}_1[g](\omega)}\xi(\d\omega)$ and $\int_\rbb \overline{\mathcal{F}_2[g](\omega)}\xi(\d\omega)$ are equal a.s. We therefore conclude that $V$ is well-defined.
\end{proof}

The function $\widehat{r}$ from \eqref{eqn:V-op-defined-1} is called {\it the spectral density} of $\Phi$.

\begin{definition}[The function-valued version of a stationary random distribution and its integral] \label{defn:form:V(t)}
Let $\delta_t$ be the Dirac $\delta$-distribution centered at $t$. If $\delta_t$ and $1_{[0,t]}$ are in $\text{\emph{Dom}}(\Phi)$, then we define 
\begin{equation} \label{form:V(t)}
V(t):=\langle \Phi,\delta_t\rangle,\ \text{and}\ X(t):=\langle \Phi,1_{[0,t]}\rangle.
\end{equation}	
\end{definition}
Note that $X(t)$ can be well-defined without $V(t)$. 

The relationship between $V(t)$ and $\nu$ is characterized as follows.
\begin{lemma} \label{lem:finite-nu} Let $\{\Phi(g)\}_{g \in \Sc'}$ be an extended stationary random distribution with spectral measure $\nu$. Then the associated stationary random process $\{V(t)\}_{t \in \rbb}$ (as in Definition \ref{defn:form:V(t)}) is well-defined if and only if $\nu$ is a finite measure. In this situation, $X(t)=\int_0^t V(s)\d s$.
\end{lemma}
\begin{proof}
The fact that the measure $\nu$ is finite is equivalent to 
\begin{displaymath}
\F{\delta_t}(\omega)=e^{-it\omega}\in L^2(\nu)
\end{displaymath}
since $\int_\rbb\left|e^{-it\omega}\right|^2\nu(\d\omega)=\int_\rbb \nu(\d\omega)<\infty $. This is precisely the condition for $\delta_t\in \text{Dom}(V)$, which implies that $V(t)$ is well-defined. 

Let $\xi$ be the random measure with respect to $\nu$ in Definition \ref{defn:random-measure}. We note that the random measure $\xi$ satisfies the orthogonal increments: for $t_1<t_2\leq t_3<t_4$,
\begin{displaymath}
\begin{aligned}
\E{\left(\xi\left(t_4\right)-\xi\left(t_3\right)\right)\overline{\left(\xi\left(t_2\right)-\xi\left(t_1\right)\right)}} &= \nu\left(\{t_4\}\cap\{t_2\}\right)-\nu\left(\{t_4\}\cap\{t_1\}\right)\\
&\quad+\nu\left(\{t_3\}\cap\{t_2\}\right)-\nu\left(\{t_3\}\cap\{t_1\}\right) = 0,
\end{aligned}	
\end{displaymath}
since $\nu$ is assumed to be absolute continuous with respect to Lebesgue measure. Consequently, $V(t)$ is actually a stationary Gaussian process. Indeed, thanks to the characterization Theorem~\ref{thm:spectral}, we have
\begin{displaymath}
V(t)=\langle V,\delta_t\rangle= \int_\rbb \overline{e^{-it\omega}}\xi(\d\omega) = \int_\rbb e^{it\omega}\xi(\d\omega).	
\end{displaymath}
Finally, the process $X(t)$ is given by
\begin{multline} \label{form:X(t)-a}
X(t)=\int_\rbb \overline{\F{1_{[0,t]}}(\omega)}\xi(\d\omega)=\int_\rbb \overline{\int_0^te^{-is\omega}\d s}\xi(\d\omega)\\=\int_\rbb \int_0^te^{is\omega}\d s\xi(\d\omega)=\int_0^t V(s)\d s.
\end{multline}
The proof is thus complete.
\end{proof}

In general, one may understand $\langle V,g\rangle$ formally as the integral $\int_\rbb V(t)g(t)\d t$,
\[\langle V,g\rangle=\int_\rbb V(t)g(t)\d t = \int_\rbb \int_\rbb e^{it\omega}\xi(\d\omega) g(t)\d t = \int_\rbb \overline{\widehat{g}(\omega)}\xi(\d\omega)=\int_\rbb \overline{\F{g}(\omega)}\xi(\d\omega). \]
%
%

\subsection{Sufficiency of Conditions \eqref{cond1}, \eqref{cond2a}, and \eqref{cond2b}}
\label{sec:prelim:improper}
In this section, we establish that the conditions listed in Assumption \ref{a:K} are sufficient for a function to be the covariance distribution of a stationary random distribution. In Lemma \ref{lem:fourier}, we show that the improper Fourier sine and cosine transforms are well-defined for our class of memory kernels. Then, in Proposition \ref{prop:Fourier.S'}, we show that our class of memory kernels are tempered distributions and express their Fourier transform in $\Sc'$ in terms of the improper Fourier cosine transform.


\begin{lemma}  \label{lem:fourier}
Suppose that $f$ satisfies Conditions \eqref{cond1b} and \eqref{cond1c} of Assumption \ref{a:K}. Then, for $\omega\neq 0$, the improper integrals $\Fcos(\omega)=\int_0^\infty f(t)\cos(t\omega) \d t$ and $\Fsin(\omega)=\int_0^\infty f(t)\sin(t\omega) \d t$ are well-defined, continuous in $\omega$, and
 \begin{equation} \label{eq:Fcos-omega-to-infinity}
 \lim_{\omega\to\infty}\Fcos(\omega)=\lim_{\omega\to\infty} \Fsin(\omega)=0.	
 \end{equation}
\end{lemma}
\begin{proof} The proof is essentially based on that of Lemma 1 from \cite{soni1975slowly}. We rewrite it here because some of the estimates will be useful later. Fix $A>0$ large enough such that $f(t)$ decreases on $[A,\infty)$,
\begin{equation*}
\int_0^\infty f(t)\cos(t\omega)\d t = \int_0^A f(t)\cos(t\omega)\d t +\int_A^\infty f(t)\cos(t\omega)\d t .
\end{equation*}
Because $f\in L^1_{\text{loc}}(\rbb)$, the first integral on the RHS above is finite. Since $f>0$ is decreasing on $t\geq A$, using the Second Mean Value Theorem, we have that for some $z\in(A,B)$
\begin{equation*}
\begin{aligned}
\int_A^B f(t)\cos(t\omega)\d t &= f(A)\int_A^z \cos(t\omega)\d t+f(B)\int_z^B \cos(t\omega)\d t\\
&\leq  f(A)\left|\int_A^z \cos(t\omega)\d t\right|+f(B)\left|\int_z^B \cos(t\omega)\d t\right| \leq  f(A)\frac{4}{\omega},
\end{aligned}
\end{equation*}
implying
\begin{equation} \label{ineq:fourier.1a}
\left|\int_A^\infty f(t)\cos(t\omega)\d t\right| \leq\frac{4f(A)}{\omega}.
\end{equation}
Since $f(A)\downarrow 0$ as $A\rightarrow\infty$, 
\[\lim_{A\to\infty}\left|\int_A^\infty f(t)\cos(t\omega)\d t\right|=0.\]
It follows that $\int_0^\infty f(t)\cos(t\omega)\d t$ converges for all $\omega >0$.

To demonstrate continuity, consider the limit as $\omega \to \omega_0 > 0$. Using Inequality \eqref{ineq:fourier.1a} gives
\begin{displaymath}
\begin{aligned}
&\left|\int_0^\infty f(t)\left[\cos(t\omega)-\cos(t\omega_0)\right]\d t \right| \\
& \qquad \leq \left|\int_0^A f(t)\left[\cos(t\omega)-\cos(t\omega_0)\right]\d t \right| + \frac{4f(A)}{\omega}+\frac{4f(A)}{\omega_0}.	
\end{aligned}	
\end{displaymath}
Since $f\in L^1_{\text{loc}}(\rbb)$, by the Dominated Convergence Theorem, the integral on the RHS above converges to 0. Thus,
\begin{equation*}
\limsup_{\omega\to\omega_0}\left|\int_0^\infty f(t)\left[\cos(t\omega)-\cos(t\omega_0)\right]\d t \right| \leq \frac{8f(A)}{\omega_0}.
\end{equation*}
Since $A$ is arbitrarily large and $f\downarrow 0$, the continuity is evident.  Likewise, $\Fsin(\omega)$ is also well-defined and continuous for $\omega\neq 0$.

Finally, to demonstrate \eqref{eq:Fcos-omega-to-infinity}, observe that \eqref{ineq:fourier.1a} implies
\begin{equation} \label{ineq:fourier.1}
\left|\int_0^\infty f(t)\cos(t\omega)\d t \right|< \left|\int_0^A f(t)\cos(t\omega)\d t \right| + \frac{4f(A)}{\omega}.
\end{equation}
By the Riemann Lebesgue lemma, the first integral on the RHS above tends to 0 as $\omega\rightarrow\infty$. Since $f(A)$ is fixed, the second term also converges to 0, which demonstrates \eqref{eq:Fcos-omega-to-infinity}.
\end{proof}

\begin{proposition} \label{prop:Fourier.S'} Let $f$ satisfy \eqref{cond1} and either $f\in L^1(\rbb)$ or there exists $\alpha\in(0,1)$ such that $t^\alpha f(t)$ is bounded near infinity. Then $f$ is a tempered distribution and
\begin{enumerate}[(a)]
\item \label{prop:Fourier.S':a} The Fourier transform of $f$ in $\Sc'$ is given by 
$$\F{f}=\widehat{f}=2\Fcos(\omega).$$
\item \label{prop:Fourier.S':b} For any $\varphi\in\Sc$, the Fourier transform of $f^+*\varphi$ in $\Sc'$ is given by
  $$\F{f^+*\varphi}(\omega) = \widehat{f^+}\widehat{\varphi}= \left(\Fcos(\omega)-i\Fsin(\omega)\right)\widehat{\varphi}(\omega),$$
where  $f^+(t)=f(t)1_{[0,\infty)}(t)$.
\end{enumerate}
 
\end{proposition}
\begin{proof}
The statement is straightforward when $f\in L^1(\rbb)$. We are interested in the case when $t^\alpha f(t)$ is bounded near infinity. The proof is based on that of Theorem 1 from \cite{soni1975slowly}. 

\noindent \eqref{prop:Fourier.S':a} Since $f$ is locally integrable and decays to zero, it is clear that $f\in \Sc'$. We are left to show that for any $\phi\in\Sc$, there holds
\begin{equation}\label{eqn:Fourier.S'}
\int_\rbb f(t)\widehat{\phi}(t)\d t = \int_\rbb 2\Fcos(\omega)\phi(\omega)\d\omega.
\end{equation} 
For $k\in \nbb$, put $f_k(t)=f(t)1_{[0,k]}(|t|)$. We observe that $f_k\in L^1(\rbb)$, which implies
\begin{equation}\label{eqn:Fourier.S':1}
\int_\rbb f_k(t)\widehat{\phi}(t)\d t=\int_\rbb \widehat{f_k}(\omega)\phi(\omega)\d\omega.
\end{equation}
On one hand, as $k\to\infty$, $f_k(t)\widehat{\phi}(t)$ converges point-wisely to $f(t)\widehat{\phi}(t)$ and is dominated by $\left|f\widehat{\phi}\right|$. We obtain, by the Dominated Convergence Theorem,
\begin{equation}\label{eqn:Fourier.S':2}
\int_\rbb f_k(t)\widehat{\phi}(t)\d t\overset{k\to\infty}{\longrightarrow}\int_\rbb f(t)\widehat{\phi}(t)\d t.
\end{equation}
On the other hand, it is clear from the proof of Lemma \ref{lem:fourier} that, for all $\omega$ non zero, $\widehat{f_k}(\omega)\phi(\omega)$ converges to $2\Fcos(\omega)\phi(\omega)$. We are left to find a dominating function for $\widehat{f_k}$. To this end, there are two cases: $\omega>1$ and $0<\omega\leq 1$. We fix $A$ such that $f(t)$ is decreasing on $t\in[A.\infty)$.

\textit{Case 1:} $\omega>1$. We note that \eqref{ineq:fourier.1a} still holds for $f(t)1_{[0,k]}(t)$ since $f(t)1_{[0,k]}(t)$ is decreasing on $t\in[A,\infty)$. We then estimate
\begin{equation}\label{eqn:Fourier.S':3}
\left|\widehat{f_k}(\omega)\right|=2\left|
\int_0^\infty f(t)1_{[0,k]}(t)\cos(\omega t)\d t\right|\leq 2\int_0^A f(t)\d t+\frac{8f(A)}{\omega}.
\end{equation}

\textit{Case 2:} $0<\omega\leq 1$.  We split the integral $\int_0^\infty f(t)1_{[0,k]}(t)\cos(t\omega)\d t$ into three parts
\begin{eqnarray*}
\int_0^\infty f(t)1_{[0,k]}(t)\cos(\omega t)\d t &=& \int_0^1+\int_1^{A/\omega}+\int_{A/\omega}^\infty f(t)1_{[0,k]}(t)\cos(t\omega)\d t\\
&=&I_0^k(\omega)+I_1^k(\omega)+I_2^k(\omega).
\end{eqnarray*}
For $I_0^k(\omega)$, we estimate
\begin{equation}\label{eqn:Fourier.S':4}
\left|I_0^k(\omega)\right| = \left|\int_0^1 f(t)1_{[0,k]}(t)\cos(\omega t)\d t\right|
\leq \int_0^1 f(t)\d t.
\end{equation}
Next, by changing variable $z=t\omega$, we have
\begin{eqnarray*}
I_1^k(\omega)&=&\frac{1}{\omega^{1-\alpha}}\int_\omega^A \left(\frac{z}{\omega}\right)^\alpha f\left(\frac{z}{\omega}\right)1_{[0,k]}\left(\frac{z}{\omega}\right)\frac{\cos(z)}{z^\alpha}\d z\\
&=&\frac{1}{\omega^{1-\alpha}}\int_0^A 1_{[\omega,A]}(z) \left(\frac{z}{\omega}\right)^\alpha f\left(\frac{z}{\omega}\right)1_{[0,k]}\left(\frac{z}{\omega}\right)\frac{\cos(z)}{z^\alpha}\d z.
\end{eqnarray*}
Since $f(t)$ is continuous, $t^\alpha f(t)$ is bounded on $t\in[1,\infty)$. It follows that $I_1^k(\omega)$ is bounded by
\begin{equation}\label{eqn:Fourier.S':5}
I_1^k(\omega) = \frac{1}{\omega^{1-\alpha}}\int_0^A 1_{[\omega,A]}(z) \left(\frac{z}{\omega}\right)^\alpha f\left(\frac{z}{\omega}\right)1_{[0,k]}\left(\frac{z}{\omega}\right)\frac{\cos(z)}{z^\alpha}\d z \leq \frac{c}{\omega^{1-\alpha}}\int_0^A \frac{1}{z^\alpha}\d z,
\end{equation}
where $c>0$ is a constant independent with $k$ and $\omega$. Lastly, for $I_2^k(\omega)$, we invoke \eqref{ineq:fourier.1a} to find
\begin{multline}\label{eqn:Fourier.S':6}
I_2^k(\omega) = \int_{A/\omega}^\infty f(t)1_{[0,k]}(t)\cos(t\omega)\d t\leq \frac{4}{\omega}f\left(\frac{A}{\omega}\right)1_{[0,k]}\left(\frac{A}{\omega}\right)\\
= \frac{4A^\alpha}{\omega^{1-\alpha}}\left(\frac{A}{\omega}\right)^\alpha f\left(\frac{A}{\omega}\right)1_{[0,k]}\left(\frac{A}{\omega}\right)\leq \frac{a_2}{\omega^{1-\alpha}},
\end{multline}
where in the last implication, we have employed again the fact that $t^\alpha f(t)$ is bounded on $[1,\infty)$. We now combine \eqref{eqn:Fourier.S':3}, \eqref{eqn:Fourier.S':4}, \eqref{eqn:Fourier.S':5} and \eqref{eqn:Fourier.S':6} to infer the existence of constants $c_1,c_2,c_3>0$ independent with $k$ and $\omega\neq 0$ such that
\begin{equation*}
\left|\widehat{f_k}(\omega)\right|\leq  \frac{c_1}{\omega^{1-\alpha}}1_{\{|\omega|\leq 1\}}(\omega)+\frac{c_2}{\omega}1_{\{|\omega|>1\}}(\omega)+c_3.
\end{equation*}
Multiplying both sides of the above inequality by $\left|\phi(\omega)\right|$ yields
\begin{equation}\label{eqn:Fourier.S':7}
\left|\widehat{f_k}(\omega)\phi(\omega)\right| \leq \left(\frac{c_1}{\omega^{1-\alpha}}1_{\{|\omega|\leq 1\}}(\omega)+\frac{c_2}{\omega}1_{\{|\omega|>1\}}(\omega)+c_3\right)|\phi(\omega)|.
\end{equation}
We observe now that the above RHS is integrable, which implies, by The Dominated Convergence Theorem that
\begin{equation}\label{eqn:Fourier.S':8}
\lim_{k\to\infty}\int_\rbb \widehat{f_k}(\omega)\phi(\omega)\d\omega=\int_\rbb 2\Fcos(\omega)(\omega)\phi(\omega)\d\omega.
\end{equation}
We therefore infer \eqref{eqn:Fourier.S'} from \eqref{eqn:Fourier.S':1}, \eqref{eqn:Fourier.S':2} and \eqref{eqn:Fourier.S':8}.
\bigskip

\noindent \eqref{prop:Fourier.S':b} Similar to part \eqref{prop:Fourier.S':a}, the Fourier transform of $f^+$ is given by
\[\F{f^+}(\omega) = \widehat{f^+}=\Fcos(\omega)-i\Fsin(\omega).\]
Now, for $\psi\in\Sc$,
\[\langle f^+*\phi,\widehat{\psi} \rangle=\int_\rbb \int_0^\infty f(s)\phi(t-s)\d s \widehat{\psi}(t)\d t.\]
In order to switch the order of integration, we have to check Fubini Condition. Using the fact that $f$ eventually decreases, we have
\begin{align*}
\MoveEqLeft[4] \int_\rbb \int_0^\infty f(s)\left|\phi(t-s)\right|\d s \left|\widehat{\psi}(t)\right|\d t \\&= \int_\rbb \left(\int_0^A f(s)\left|\phi(t-s)\right|\d s +\int_A^\infty f(s)\left|\phi(t-s)\right|\d s\right)\left|\widehat{\psi}(t)\right|\d t \\
&\leq \int_\rbb \left(\|\phi\|_{L^\infty}\int_0^A f(s)\d s +f(A)\|\phi\|_{L^1}\right)\left|\widehat{\psi}(t)\right|\d t \\
&=\|\widehat{\psi}\|_{L^1}\left(\|\phi\|_{L^\infty}\int_0^A f(s)\d s +f(A)\|\phi\|_{L^1}\right).
\end{align*}
We thus obtain
\begin{displaymath}
\langle f^+*\phi,\widehat{\psi} \rangle=\langle f^+,\widetilde{\phi}*\widehat{\psi} \rangle=\langle \widehat{f^+},\widehat{\phi}\psi\rangle=\langle \widehat{f^+}\widehat{\phi},\psi\rangle,
\end{displaymath}
which completes the proof.
\end{proof}

\begin{corollary} \label{cor:covariance-dis} Suppose $f:\rbb\to\rbb$ satisfies the Hypothesis of Proposition \ref{prop:Fourier.S'}, then $f$ is the covariance distribution of a stationary random distribution whose spectral density is $2 \Fcos(\omega)$.
\end{corollary}
\begin{proof}
It suffices to check that $f(t)$ is positive definite as a tempered distribution. For $\varphi\in \Sc$, we have
\[\langle f,\varphi*\widetilde{\varphi} \rangle = \int_\rbb f(t)\left(\varphi*\widetilde{\varphi}\right)(t)\d t=\int_\rbb 2\Fcos(\omega)\left|\varphi(\omega)\right|^2 \d\omega\geq 0,\]
where the second and third implications follow from Proposition \ref{prop:Fourier.S'} \eqref{prop:Fourier.S':a} and Condition \eqref{cond1d}, respectively.
\end{proof}

\subsection{Examples of admissible memory kernels}

The Condition \eqref{cond1} requires that the Fourier cosine transform $\Fcos$ be positive. A sufficient condition for $K$ to guarantee positive Fourier cosine is that $K(t)$ be convex and locally integrable on $[0,\infty)$. To be precise, we record the following Lemma, whose proof can be found in \cite{tuck2006positivity}.
\begin{lemma}\label{lem:fourier-positive} Suppose that $f \in L^1_{\text{loc}}([0,\infty))$, convex on $(0,\infty)$ and decreasing to zero as $t\rightarrow\infty$, then for all $\omega\neq 0$,
\begin{displaymath}
\int_0^\infty \!\! f(t)\cos(t\omega)\d t=\lim_{t\to\infty}\int_0^t f(s)\cos(s\omega)\d s >0.	
\end{displaymath}
\end{lemma} 

\subsubsection{Sums of exponential functions} 
\label{sec:prelim:exponentials}

One family ofmemory functions that has proved useful in statistical analysis of viscoelastic diffusion is the Generalized Rouse kernels \cite{mckinley2009transient, lysy2016model}. While these functions can have arbitrarily many terms, the family is fully described by three parameters, which makes the associated GLE amenable for parameter inference \cite{lysy2016model}. Let $p \geq 1$, $N \in \nbb$ and $\tau_0 > 0$ be given. Then we define the Generalized Rouse kernels to be the set of functions of the
\begin{equation}\label{form:K-rouse-N}
K_N(t; p, \tau_0) = \frac{1}{N} \sum_{k=0}^{N-1} e^{-|\frac{t}{\tau_0}|\left(\frac{k}{N}\right)^p}.
\end{equation}
There is in fact an explicit form for the limit as $N$ tends to infinity:
\begin{equation}\label{form:K-rouse}
K_\text{Rouse}(t; p, \tau_0) := \lim_{N\to\infty}\frac{1}{N}\sum_{k=0}^{N-1} e^{-|\frac{t}{\tau_0}|\left(\frac{k}{N}\right)^p}=\int_0^1 e^{-|\frac{t}{\tau_0}|x^p}\d x.
\end{equation}
The following proposition asserts that as $N$ becomes larger, the tail of $K_N(t; p, \tau_0)$ behaves more and more like a power law of the form $t^{-1/p}$.
\begin{proposition} Suppose that $p\geq 1,N\in\nbb,\tau_0>0$. Denote by~$K_N=K_N(t;p,\tau_0)$ and $K=K_\text{Rouse}(t; p, \tau_0)$ where $K_N(t;p,\tau_0)$ and $K_\text{Rouse}(t;p,\tau_0)$ are as in \eqref{form:K-rouse-N} and \eqref{form:K-rouse}, respectively. Then,
\begin{enumerate}[(a)]
\item $\lim_{N\to\infty}\sup_{t\in\rbb}|K_N(t)-K(t)|=0.$
\item $K(t)\sim t^{-1/p}$ as $t\to\infty$.
\end{enumerate}  
\end{proposition}
\begin{proof}
(a) Since $K_n$ and $K$ are even, it suffices to show that
 \begin{equation}\label{lim:K-rouse-N-1}
\lim_{N\to\infty}\sup_{t\geq 0}|K_N(t)-K(t)|=0.
\end{equation}
We observe that $K_N(t),\, K(t)\in[0,1]$ and that they are monotonically decreasing to zero on $t\in[0,\infty)$. The uniform convergence then follows from the point-wise convergence, see Exercise 13, pg. 167, \cite{rudin1964principles}.

(b) For $t>0$, using a change of variable $y=\frac{t}{\tau_0}x^p$, $K(t)$ is equal to
\begin{align*}
K(t)=\frac{\tau_0^{1/p}}{pt^{1/p}}\int_0^{t/\tau_0} y^{1/p-1}e^{-y}\d y.
\end{align*}
It follows immediately that 
\begin{align*}
t^{1/p}K(t)=\frac{\tau_0^{1/p}}{p}\int_0^{t/\tau_0} y^{1/p-1}e^{-y}\d y\longrightarrow \frac{\tau_0^{1/p}}{p}\Gamma\left(\frac{1}{p}\right),\quad t\to\infty,
\end{align*}
where $\Gamma(x)$ denotes the usual gamma function evaluated at $x$.
\end{proof}



The Generalized Rouse kernel is a special case of a class of convex functions called the {\it completely monotone} functions.

\begin{definition} \label{defn:CM}
A function $K:(0,\infty)\to\rbb$ is completely monotone if $K$ is of class $C^\infty$ and $(-1)^n K^{(n)}(t)\geq 0$ for all $n\geq 0$, $t>0$.
Denote
\begin{enumerate}[(a)]
\item $\mathcal{CM}\label{CM}$ is the set of all \textit{completely monotone} function.
\item $\mathcal{CM}_b\label{CMb}$ is the set of all $K\in \mathcal{CM}\cap C[0,\infty)$.
\end{enumerate} 
\end{definition}

These functions are characterized by the following classical theorem.

\begin{theorem}[Hausdorff-Bernstein-Widder Theorem \cite{schilling2012bernstein}]\label{thm:bernstein} A function $K$ is completely monotone if and only if $K$ admits the formula
\begin{equation} \label{laplace}
K(t)=\int_0^\infty e^{-tx} \mu (\d x),
\end{equation}
where $\mu$ is a positive measure on $[0,\infty)$. 
\end{theorem}

\subsubsection{Power Law Kernels}
In \cite{kou2004generalized} and \cite{kou2008stochastic}, the author considered the kernel 
\begin{equation} \label{form:Kou-kernel}
K_H(t)=2H(2H-1)|t|^{2H-2},
\end{equation}
where $H\in(1/2,1)$. Using explicit Fourier transform of $K_H$, it is shown in \cite{kou2008stochastic} that the MSD satisfies $\E{X^2(t)}\sim t^{2-2H}$, which is subdiffusive.

To check that $K_H(t)$ verifies Assumption \ref{a:K}, we first note that $K_H(t)$ is a power-law at infinity, which is consistent with Condition \eqref{cond2b}. On the other hand, $K_H(t)$ is convex on $(0,\infty)$. Lemma~\ref{lem:fourier-positive} then implies that the improper Fourier cosine transform $\Kcos(\omega):=\int_0^\infty K_H(t)\cos(t\omega)\d t$ is positive for every non-zero $\omega$. It follows that $K_H(t)$ satisfies Condition \eqref{cond1}. We now can apply Corollary \ref{cor:covariance-dis} to see that $K_H$ is the covariance distribution of a stationary random distribution whose spectral density is $2 \Kcos(\omega)$

Our theory of weak solution in Section \ref{sec:weak-solutions}  therefore applies to $K_H$. Furthermore, our result on MSD in Section \ref{sec:msd} generalizes the result from \cite{kou2008stochastic}, namely, the class of functions satisfying \eqref{cond1} $+$ \eqref{cond2b}, of which $K_H$ is a special case,  leads to subdiffusive MSD.

%
%

\subsubsection{An example of a class of non-convex kernels}

Given the examples we have presented so far, it might appear that convexity is required of $K$ but this is not the case. In Lemma \ref{lem:non-convex-K} below, we show that our class of admissible memory kernels includes functions of the form $K(t)=\varphi\left(t^2\right)$, where $\varphi\in\mathcal{CM}_b$. Take $\beta > 0$, then $\big(1+t^2\big)^{-\beta/2}$ is a non-convex yet admissible memory kernel because $\big(1+t\big)^{-\beta/2}$ is a completely monotone function. Note that when $\beta \in (0,1)$, the associated GLE is subdiffusive. In general, let $\mu$ be the representing measure of $\varphi$ in Theorem \ref{thm:bernstein}, then $K(t) := \varphi(t^2)$ admits the representation
\begin{equation}\label{laplace1}
K(t)=\int_0^\infty e^{-t^2x}\mu(\d x).
\end{equation}
This function is not convex, but we are able to assert the following.
\begin{lemma} \label{lem:non-convex-K}
Let $K(t)=\varphi\left(t^2\right)$, where $\varphi\in\mathcal{CM}_b$. Then, for every $\omega\neq 0$, $\Kcos(\omega)>0$.
\end{lemma}
\begin{proof}
Substituting $K$ with the formula \eqref{laplace1}, we have a chain of limits
\begin{multline}\label{lim:non-cv-1}
\int_0^\infty K(t)\cos(t\omega)\d t =\lim_{A\to\infty}\int_0^A K(t)\cos(t\omega)\d t \\
=\lim_{A\to\infty}\int_0^A \int_0^\infty e^{-t^2x}\mu(\d x)\cos(t\omega)\d t=\lim_{A\to\infty}\int_0^\infty \int_0^A e^{-t^2x}\cos(t\omega)\d t\mu(\d x),
\end{multline}
where in the last equality, we use the Fubini Theorem to switch the order of integration. Now applying the Second Mean Value Theorem, we infer a $\xi\in (0,A)$ such that
\[\left|\int_0^A e^{-t^2x}\cos(t\omega)\d t\right|=\left|e^{0}\int_0^\xi\cos(t\omega)\d t\right|\leq \frac{1}{\omega}.\]
We note that the representing measure $\mu$ is finite. Hence, by the Dominated Convergence Theorem, we obtain
\begin{equation} \label{lim:non-cv-2}
\lim_{A\to\infty}\int_0^\infty \int_0^A e^{-t^2x}\cos(t\omega)\d t\mu(\d x)=\int_0^\infty \int_0^\infty e^{-t^2x}\cos(t\omega)\d t\mu(\d x).
\end{equation}
It follows from \eqref{lim:non-cv-1} and \eqref{lim:non-cv-2} that
\begin{displaymath}
\int_0^\infty K(t)\cos(t\omega)\d t=\int_0^\infty \left[\int_0^\infty e^{-t^2x}\cos(t\omega)\d t\right]\mu(\d x)=\int_0^\infty \frac{\sqrt{\pi}}{\sqrt{x}}e^{-\omega^2/4x}\mu(\d x),
\end{displaymath}
which implies that $\int_0^\infty K(t)\cos(t\omega)\d t > 0$.
\end{proof}

In anticipation of the results that follow, we remark that since functions of this form satisfy Assumption \ref{a:K} but not Condition \ref{cond:convex}, it is not clear whether the associated GLE is well-defined in the $m = \lambda = 0$ case.

\section{Abelian Theorems for Fourier Transforms}
\label{sec:tauberian}
In the subdiffusive case (with our specialized conditions), the behavior of the Fourier transform near the origin and near infinity can be characterized in a manner analogous to the Abelian theorems for the Laplace transform in the sense presented by Feller \cite{feller1966introduction}.

\begin{proposition} \label{prop:fourier.taube}
Suppose that $f$ satisfies the conditions \eqref{cond1b}, \eqref{cond1c} and \eqref{cond2b} from Assumption \ref{a:K}. Then
\begin{equation} \label{lim:taubecos}
\lim_{\omega\to 0}\omega^{1-\alpha}\Fcos(\omega)\in (0,\infty)\quad\text{and}\quad\lim_{\omega\to 0}\omega^{1-\alpha}\Fsin(\omega)\in (0,\infty).
\end{equation}
\end{proposition}
\begin{rem} Proposition \ref{prop:fourier.taube} is slightly different from Theorem 1, \cite{soni1975slowly}, in which $f$ is assumed to be finite at the origin. Our class of memory kernels need not satisfy this condition, recalling \eqref{form:Kou-kernel} for example. The technique that we use to treat the case where $K(t)$ is infinite at the origin is similar to the proof of Theorem 1.1 from \cite{inoue1995abel}.
\end{rem}
\begin{proof}[Proof of Proposition \ref{prop:fourier.taube}]
To establish \eqref{lim:taubecos}, we shall improve the proof of Theorem 1 from \cite{soni1975slowly}. Denote $c=\lim_{t\to\infty}t^\alpha f(t)$. By a change of variable, we have
\begin{eqnarray*}
\int_0^\infty f(t)\cos(\omega t)\d t &=& \int_0^\infty f\left(\frac{z}{\omega}\right)\frac{\cos(z)}{\omega}\d z.
\end{eqnarray*}
Similar to the proof of Proposition \ref{prop:Fourier.S'}(a), fixing $A$ such that $f(t)$ is decreasing on $t\in[A,\infty)$, we split the above integral in three parts
\begin{eqnarray*}
\omega^{1-\alpha}\int_0^\infty f(t)\cos(\omega t)\d t &=& \omega^{1-\alpha}\left[\int_0^\omega+\int_\omega^A+\int_A^\infty f\left(\frac{z}{\omega}\right)\frac{\cos(z)}{\omega}\d z\right]\\
&=&  \omega^{1-\alpha}\left(I_0(\omega)+I_1(\omega)+I_2(\omega)\right).
\end{eqnarray*}
For $I_0(\omega)$, changing variable again, we have
\begin{equation}\label{eqn:fourier-taube:0}
\omega^{1-\alpha}\left|I_0(\omega)\right| = \omega^{1-\alpha}\left|\int_0^1 f(t)\cos(\omega t)\d t\right|
\leq \omega^{1-\alpha}\int_0^1 f(t)\d t \overset{\omega\to 0}{\longrightarrow}0,
\end{equation}
since $f\in L^1_{loc}$. For $I_1(\omega)$, Condition \eqref{cond2b} combining with continuity implies that $t^\alpha f(t)$ is uniformly bounded on $t\in [1,\infty)$. It follows from the Dominated Convergence Theorem that
\begin{equation}\label{eqn:fourier-taube:1}
\omega^{1-\alpha} I_1(\omega) = \int_0^A 1_{[\omega,A]}(z) \left(\frac{z}{\omega}\right)^\alpha f\left(\frac{z}{\omega}\right)\frac{\cos(z)}{z^\alpha}\d z \rightarrow c\int_0^A \frac{\cos(z)}{z^\alpha}\d z,
\end{equation}
where $c=\lim_{t\to\infty}t^\alpha f(t)$. For the last term $I_2(\omega)$, we invoke \eqref{ineq:fourier.1a} again to find
\begin{equation}\label{eqn:fourier-taube:2}
\omega^{1-\alpha} I_2(\omega) 
\leq  \frac{4}{A^\alpha}\left(\frac{A}{\omega}\right)^\alpha f\left(\frac{A}{\omega}\right),
\end{equation}
which implies
\begin{equation} \label{eqn:fourier-taube:3}
\limsup_{\omega\to 0}\left| I_2(\omega) \right| \leq c\frac{4}{A^\alpha}.
\end{equation}
Since $A$ is chosen arbitrarily large, combining \eqref{eqn:fourier-taube:0}, \eqref{eqn:fourier-taube:2} and \eqref{eqn:fourier-taube:3}, we obtain
\begin{equation} \label{eqn:fourier-taube:4}
\lim_{\omega\to 0}\omega^{1-\alpha}\Fcos(\omega)=c\int_0^\infty \frac{\cos(z)}{z^\alpha}\d z.
\end{equation}
We note that $\cos(z)/z^\alpha$ satisfies $\int_0^\infty \cos(z) z^{-\alpha} \d z\in (0,\infty)$, see \cite{tuck2006positivity}. We therefore obtain the Fourier cosine limit in \eqref{lim:taubecos}. 
The Fourier sine limit is established using a similar argument.
\end{proof}

On the other hand, for the asymptotic behavior of Fourier transform, we require Assumption \ref{cond:cv}. This assumption is particularly useful in Section \ref{sec:weak-solutions:mzero-lambdazero} and Section \ref{sec:msd:mzero-lambdazero}. We first observe that if a function $f$ is convex and twice differentiable, the Fourier transform has the following representation, whose proof makes use of standard technique of integration by parts.

\begin{lemma} \label{lem4}
Suppose $f\in C^2(0,\infty)$ satisfies \eqref{cond1b}, \eqref{cond1c}. Furthermore, we assume that $f$ is convex and $\lim_{t\to 0^+}tf(t)=0$. Then for all $\omega>0$,
\begin{equation} \label{eqn:lem4.1}
\int_0^\infty f(t)\cos(t\omega)\d t  = \frac{1}{\omega^2}\int_0^\infty f''(t)\left[1-\cos(t\omega)\right]\d t.
\end{equation}
\end{lemma}
\begin{proof}
Since for all $t>0$, $f(t)$ is decreasing and $f''(t)\geq 0$ , $f'(t)$ is increasing and negative. Now integration by parts gives
\begin{equation}\label{eqn:lem4-1}
\begin{aligned}
\int_0^\infty f(t)\cos(t\omega)\d t & = f(t)\frac{\sin(t\omega)}{\omega}\bigg|_{t\to 0}^{t\to\infty}-\frac{1}{\omega}\int_0^\infty f'(t)\sin(t\omega)\d t\\
&=-\frac{1}{\omega}\int_0^\infty f'(t)\sin(t\omega)\d t,
\end{aligned}
\end{equation}
since $\lim_{t\to 0^+}f(t)\sin(t\omega)/\omega=\lim_{t\to 0^+}tf(t)\sin(t\omega)/(t\omega)=0$ by assumption on $f$. For $t>0$, integration by part once again gives
\begin{equation}\label{eqn:lem4-1a}
\begin{aligned}
-\frac{1}{\omega}\int_t^\infty f'(x)\sin(x\omega)\d x &= f'(x)\frac{\cos(x\omega)}{\omega^2}\bigg|_{x=t}^{x\to\infty}-\frac{1}{\omega^2}\int_t^\infty f''(x)\cos(x\omega)\d x\\
&=-f'(t)\frac{\cos(t\omega)}{\omega^2}-\frac{1}{\omega^2}\int_t^\infty f''(x)\cos(x\omega)\d x\\
&=-f'(t)\frac{\cos(t\omega)-1}{\omega^2}+\frac{1}{\omega^2}\int_t^\infty f''(x)\left[1-\cos(x\omega)\right]\d x.
\end{aligned}
\end{equation}
Sending $t\rightarrow 0$, we have indeed
\[\lim_{t\to 0}f'(t)\frac{\cos(t\omega)-1}{\omega^2}=\lim_{t\to 0}t^2f'(t)\frac{\cos(t\omega)-1}{(t\omega)^2}=0.\]
To see that, we use the fact that $f'<0$ and is decreasing on $t\in(0,\infty)$ to find
\begin{equation}\label{eqn:lem4-2}
0\geq t^2f'(t)=2t\int_{t/2}^{t} f'(t)\d x\geq 2t\int_{t/2}^{t} f'(x)\d x=2t\left(f(t)-f(t/2)\right)\overset{t\to 0}{\longrightarrow}0. 
\end{equation}
It follows from \eqref{eqn:lem4-1a} that
\begin{equation}\label{eqn:lem4-3}
-\frac{1}{\omega}\int_0^\infty f'(x)\sin(x\omega)\d x =
\frac{1}{\omega^2}\int_0^\infty f''(x)\left[1-\cos(x\omega)\right]\d x.
\end{equation}
Finally, \eqref{eqn:lem4.1} follows from \eqref{eqn:lem4-1} and \eqref{eqn:lem4-3}, which concludes the proof.
\end{proof}
We finally turn to asymptotic behavior of Fourier transform. The following proposition is useful in Section \ref{sec:weak-solutions:mzero-lambdazero} where there is neither mass nor viscous drag.
\begin{proposition} \label{prop:fourier-cv}
Suppose that $f(t)$ satisfies \eqref{cond1b} $+$ \eqref{cond:convex}.
\begin{enumerate}[(a)]
\item If $f(t)$ further satisfies \eqref{cond:cv-finite}. Then, 
\begin{equation} \label{lim:fourier-cv-1}
\lim_{\omega\to\infty}\omega^{2-\sigma_1}\Fcos(\omega)=0,\ \text{and}\ \lim_{\omega\to\infty}\omega\Fsin(\omega)=f(0),
\end{equation}
where $\sigma_1$ is the exponent from \eqref{cond:cv-finite}.
\item If $f(t)$ further satisfies \eqref{cond:cv-infinite}. Then, 
\begin{equation} \label{lim:fourier-cv-2}
\lim_{\omega\to\infty}\omega^{1-\sigma_2}\Fcos(\omega)\in(0,\infty),\ \text{and}\ \lim_{\omega\to\infty}\omega^{1-\sigma_2}\Fsin(\omega)\in(0,\infty),
\end{equation}
where $\sigma_2$ is the power constant from \eqref{cond:cv-infinite}.
\end{enumerate}
\end{proposition}
\begin{proof} (a) Since $f(t)$ is convex on $t\in(0,\infty)$, it follows from Lemma \ref{lem4} that
		\begin{equation}\label{eqn:fourier-cv-1}
		\Fcos(\omega)=\frac{1}{\omega^2}\int_0^\infty f''(t)\left(1-\cos(t\omega)\right)\d t.
		\end{equation}
		By changing of variable $z=t\omega$, \eqref{eqn:fourier-cv-1} is equivalent to
		\begin{equation}\label{eqn:fourier-cv-2}
		\omega^{2-\sigma_1}\Fcos(\omega)=\int_0^\infty \left(\frac{z}{\omega}\right)^{1+\sigma_1}f''\left(\frac{z}{\omega}\right)\frac{1-\cos(z)}{z^{1+\sigma_1}}\d z.
		\end{equation}
		We aim to use the Dominated Convergence Theorem on the RHS above. Indeed, the integrand is dominated by $\frac{1-\cos(z)}{z^{1+\sigma_1}}$, which is integrable. To see that, we claim that $t^{1+\sigma_1}f''(t)$ is uniformly bounded on $t\in(0,\infty)$. The only concerns are when $t$ is near zero and when $t$ is large. On one hand, notice that $f''$ is monotone near the origin by condition \eqref{cond:convex}. We write
		\begin{equation*} -t^{\sigma_1} f'(t) =t^{\sigma_1}\int_t^{2t}f''(s)\d s-t^{\sigma_1} f'(2t)\geq t^{1+\sigma_1}f''(t)-t^{\sigma_1} f'(2t),
		\end{equation*}
where we have assumed $f''(t)$ is increasing near the origin. By shrinking $t$ to zero, we obtain $t^{1+\sigma_1} f''(t)\rightarrow 0$. Similar estimate also applies if we assume $f''(t)$ is decreasing, namely
\begin{equation*} -t^{\sigma_1} f'(t) =t^{\sigma_1}\int_t^{2t}f''(s)\d s-t^{\sigma_1} f'(2t)\geq t^{1+\sigma_1}f''(2t)-t^{\sigma_1} f'(2t).
		\end{equation*}
On the other hand, as $t\to\infty$, $f(t)/t^{1-\sigma_1}\rightarrow 0$. We employ the same trick to see that
		\begin{equation*}
		-\frac{f(t)}{t^{1-\sigma_1}}=\frac{\int_1^t-f'(s)\d s}{t^{1-\sigma_1}}-\frac{f(1)}{t^{1-\sigma_1}}\geq -\frac{t-1}{t^{1-\sigma_1}}f'(t)-\frac{f(1)}{t^{1-\sigma_1}},
		\end{equation*} 
		since $-f'(t)$ is increasing on the positive half line. By taking $t\to\infty$, we obtain $t^{\sigma_1} f'(t)\rightarrow 0$. L'Hospital Rule then implies
		\begin{equation*}
		\lim_{t\to\infty}\frac{f''(t)}{-\sigma_1 t^{-1-\sigma_1}}=\lim_{t\to\infty}\frac{f'(t)}{t^{-\sigma_1}}=0.
		\end{equation*}
		Now, from \eqref{eqn:fourier-cv-2}, sending $\omega$ to infinity, it follows from the Dominated Convergence Theorem that
		\begin{equation} \label{eqn:fourier-cv-3}
		\lim_{\omega\to\infty}\omega^{2-\sigma_1}\Fcos(\omega)=0
		\end{equation}
		For the Fourier sine transform, we integrate by parts to find
		\begin{equation}\label{eqn:fourier-cv-4}
		\Fsin(\omega)=\int_0^\infty f(t)\sin(t\omega)\d t = \frac{f(0)}{\omega}+\int_0^\infty f'(t)\frac{\cos(t\omega)}{\omega}\d t.
		\end{equation}
		Multiplying through by $\omega$, we obtain
		\begin{equation}\label{eqn:fourier-cv-5}
		\omega\Fsin(\omega) = f(0)+\int_0^\infty f'(t)\cos(t\omega)\d t.
		\end{equation}
		It suffices to show that $\lim_{\omega\to\infty}\int_0^\infty f'(t)\cos(t\omega)\d t=0$. This in turn follows immediately from  Lemma \ref{lem:fourier} since $f'\in L^1(0,\infty)$ and $-f'(t)\downarrow 0$ as $t\to\infty$.
		
(b) We use \eqref{eqn:fourier-cv-2} again to write further
		\begin{equation}\label{eqn:fourier-cv-6}
		\begin{aligned}
		\omega^{1-\sigma_2}\Fcos(\omega)&=\int_0^\infty \left(\frac{z}{\omega}\right)^{2+\sigma_2}f''\left(\frac{z}{\omega}\right)\frac{1-\cos(z)}{z^{2+\sigma_2}}\d z\\
		&=\int_0^\omega+\int_\omega^\infty \left(\frac{z}{\omega}\right)^{2+\sigma_2}f''\left(\frac{z}{\omega}\right)\frac{1-\cos(z)}{z^{2+\sigma_2}}\d z\\
		&=I_1(\omega)+I_2(\omega).
		\end{aligned}
		\end{equation}
		We first claim that $\lim_{\omega\to \infty}I_2(\omega)=0$. Indeed, by changing variable again $t=z/\omega$, we have
		\begin{equation}\label{eqn:fourier-cv-7}
		\begin{aligned}
		I_2(\omega)&=\frac{1}{\omega^{1+\sigma_2}}\int_1^\infty f''(t)\left(1-\cos(t\omega)\right)\d t\\
		&\leq\frac{2}{\omega^{1+\sigma_2}}\int_1^\infty f''(t)\d t  =\frac{-2f'(1)}{\omega^{1+\sigma_2}}\to 0.			
		\end{aligned}	
		\end{equation}
		For $I_1(\omega)$, we write 
\begin{equation}\label{eqn:fourier-cv-8}
I_1(\omega)=\int_0^\infty 1_{[0,\omega]}(z) \left(\frac{z}{\omega}\right)^{2+\sigma_2}f''\left(\frac{z}{\omega}\right)\frac{1-\cos(z)}{z^{2+\sigma_2}}\d z.
\end{equation}		
We wish to obtain from the Dominated Convergence Theorem that
\begin{equation}\label{eqn:fourier-cv-9}
\lim_{\omega\to\infty}I_1(\omega)=\int_0^\infty\frac{1-\cos(z)}{z^{2+\sigma_2}}\d z\times\lim_{t\to 0} t^{2+\sigma_2}f''\left(t\right)\in(0,\infty)
\end{equation}		
To that end, we	claim that $\lim_{t\to 0}t^{2+\sigma_2}f''(t)\in(0,\infty)$ and that $t^{2+\sigma_2}f''(t)$ is uniformly bounded on $t\in(0,1]$. The latter follows immediately from the former claim and the fact that $t^{2+\sigma_2}f''(t)$ is continuous. By condition \eqref{cond:cv-infinite}, $f(0^+)=\infty$. We apply L'Hospital Rule twice to see that
\begin{equation} \label{eqn:fourier-cv-9a}
\lim_{t\to 0}\frac{f(t)}{t^{-\sigma_2}}=\lim_{t\to 0}\frac{-f'(t)}{t^{-1-\sigma_2}}=\lim_{t\to 0}\frac{f''(t)}{t^{-2-\sigma_2}}\in(0,\infty).
\end{equation}
Thus, the integrand in \eqref{eqn:fourier-cv-8} is dominated by 
\[\sup_{t\in(0,1]}t^{2+\sigma_2}f''(t) \,\frac{1-\cos(z)}{z^{2+\sigma_2}},\]
as a function of $z$, which is integrable. We thus have shown \eqref{eqn:fourier-cv-9}. We finally combine \eqref{eqn:fourier-cv-7}, \eqref{eqn:fourier-cv-9} with \eqref{eqn:fourier-cv-6} to obtain
\begin{equation}\label{eqn:fourier-cv-10}
\lim_{\omega\to\infty}\omega^{1-\sigma_2}\Fcos(\omega)\in(0,\infty).
\end{equation}
For the Fourier sine transform, we integrate by part to obtain
\begin{equation}\label{eqn:fourier-cv-11}
\Fsin(\omega)=\frac{1}{\omega}\int_0^\infty-f'(t)\left(1-\cos(t\omega)\right),
\end{equation}
implying
\begin{equation}\label{eqn:fourier-cv-12}
\begin{aligned}
\omega^{1-\sigma_2}\Fsin(\omega)&=\int_0^\infty-\left
(\frac{z}{\omega}\right)^{1+\sigma_2}f'\left(\frac{z}{\omega}\right)\left(\frac{1-\cos(z)}{z^{1+\sigma_2}}\right)\d z\\
&=\int_0^\omega+\int_\omega^\infty-\left
(\frac{z}{\omega}\right)^{1+\sigma_2}f'\left(\frac{z}{\omega}\right)\left(\frac{1-\cos(z)}{z^{1+\sigma_2}}\right)\d z\\
&=I_3(\omega)+I_4(\omega).
\end{aligned}
\end{equation}
For $I_4(\omega)$, similar to \eqref{eqn:fourier-cv-7}, we have the chain of implications
\begin{equation}\label{eqn:fourier-cv-13}
\begin{aligned}
I_4(\omega)&=\frac{1}{\omega^{\sigma_2}}\int_1^\infty-f'(t)\left(1-\cos(t\omega)\right)\d t \leq \frac{1}{\omega^{\sigma_2}}\int_1^\infty-f'(t)\d t= \frac{f(1)}{\omega^{\sigma_2}}\rightarrow 0.	
\end{aligned}
\end{equation}
For $I_3(\omega)$, similar to the argument used to establish \eqref{eqn:fourier-cv-9}, we observe that $\lim_{t\to 0^+}t^{1+\sigma_2}f'(t)\in(0,\infty)$ thanks to \eqref{eqn:fourier-cv-9a} and that $t^{1+\sigma_2}f'(t)$ is bounded on $t\in(0,1]$ thanks to continuity. The Dominated Convergence Theorem then implies
\begin{equation}\label{eqn:fourier-cv-14}
\lim_{\omega\to\infty}\omega^{1-\sigma_2}\Fsin(\omega)\in(0,\infty).
\end{equation}
The proof is complete.
\end{proof}

\section{Weak solutions for the Generalized Langevin Equation}
\label{sec:weak-solutions}

In order to define our notion of weak solutions for the GLE, we multiply \eqref{eq:gle} through by a test function $\f \in \Sc$ and integrate over the real line with respect to time. Formally, if we integrate by parts on the left-hand side and perform a change-of-variables in the convolution term, we arrive at the integral equation  
\begin{displaymath}
\begin{aligned}
	- m\int_\rbb V(t)\varphi'(t)\d t &= - \lambda \int_{\rbb} V(t) \f(t) \d t - \beta \int_\rbb V(t)\int_\rbb K^+(u)\varphi(t+u)\d u\d t \\
	& \, \quad + \sqrt{\beta} \int_\rbb F(t)\varphi(t)\d t + \sqrt{2\lambda} \int_{\rbb} \f(t) \d W(t),
\end{aligned}
\end{displaymath}
where we have introduced the notation $K^+\!(t) := K(t) \, 1_{\{t\geq 0\}}$. If we understand $V$, $F$, and the white noise process $\dot W$ as stationary random distributions in the sense of Section \ref{sec:prelim}, then we can write the GLE in its weak form
\begin{equation} \label{eq:gle-weak}
\langle V,-m\varphi'+\lambda\varphi+\beta\widetilde{K^+*\widetilde{\varphi}}\rangle = \sqrt{2\lambda}\langle \dot W,\varphi\rangle+\sqrt{\beta}\langle F,\varphi\rangle,
\end{equation}
where $\widetilde{f}(x) := f(-x)$. In this setting, the stationary random distributions $\dot W$ and $F$ are defined in terms of their covariance structures:
\begin{displaymath}
\E{\langle \dot W,\varphi_1\rangle \langle \dot W,\varphi_2\rangle}=\int_\rbb \varphi_1(t)\varphi_2(t)\d t, \, \, \E{\langle F,\varphi_1\rangle\overline{\langle F,\varphi_2\rangle}} = \int_\rbb K(t)\left( \varphi_1*\widetilde{\varphi}_2 \right)(t)\d t.
\end{displaymath}
In other words, the spectral measure of $\dot W$ is Lebesgue measure and the spectral measure of $F$ is $\widehat{K}(\d \omega)$. In fact, we showed in Section \ref{sec:prelim:improper} that $F$ has a spectral density, $2 \Kcos(\omega)$. See Corollary \ref{cor:covariance-dis} in particular.




\begin{definition} \label{def:weak-solution}
Let $\nu$ be a non-negative measure satisfying condition \eqref{ineq:spectral-measure} and $V$ be the operator associated with $\nu$ defined in \eqref{eqn:V.op}. Then $V$ is a  \textit{weak solution} for Equation \eqref{eq:gle-weak} if $V$ satisfies the following conditions.
\begin{enumerate}[(a)]
\item For all $\varphi\in\Sc$, $K^+*\varphi$ belongs to $\text{Dom}(V)$.
\item For any $\varphi,\psi\in \Sc$, it holds that
\begin{multline*}
\E{\langle V,-m\varphi'+\lambda\varphi+\beta\widetilde{K^+*\widetilde{\varphi}}\rangle\overline{\langle V,-m\psi'+\lambda\psi+\beta\widetilde{K^+*\widetilde{\psi}}\rangle}}\\ =\E{\langle \sqrt{2\lambda} \dot W+\sqrt{\beta}F,\varphi\rangle \overline{\langle \sqrt{2\lambda} \dot W+\sqrt{\beta}F,\psi\rangle}}.\end{multline*}
\end{enumerate}
\end{definition}

The proof that weak solutions exist is sensitive what is assumed about the parameters $m$ and $\lambda$.  We start with the most delicate proof, which is in the case $(m > 0, \lambda = 0)$. The cases $(m > 0, \lambda > 0)$ and $(m = 0, \lambda > 0)$ follow a similar argument (Sections \ref{sec:weak-solutions:mpos-lambdapos} and \ref{sec:weak-solutions:mzero-lambdapos}). The case $(m=0, \lambda = 0)$ requires further assumptions about the memory kernel and we handle this case in Section \ref{sec:weak-solutions:mzero-lambdazero}.

\subsection{Weak solutions when $m > 0$ but $\lambda = 0$}

We begin by introducing the function $\widehat{r}$ in the following Lemma.

\begin{lemma}\label{lemma:r-hat-L1}
Let $K$ satisfy Assumption \ref{a:K}. Denote
\begin{equation} \label{eqn:spectral-density}
\widehat{r}(\omega) :=\frac{\widehat{K}(\omega)}{2\pi\lvert mi\omega+\widehat{K^+}(\omega)\rvert^2}.
\end{equation}
Then $\widehat{r}$ belongs to $L^1(\rbb)$.
\end{lemma}
\begin{proof}
We can rewrite the formula for $\widehat{r}(\omega)$ as
\begin{equation} \label{r.hat}
\widehat{r}(\omega) = \frac{1}{2\pi}\times\frac{2\Kcos(\omega)}
{\left[\Kcos(\omega) \right]^2
+\left[m\omega- \Ksin(\omega) \right]^2}.
\end{equation}
Observing that $\widehat{r}$ is even, we only need to consider $\omega\in [0,\infty)$. By Lemma \ref{lem:fourier}, $\widehat{r}(\omega)$ is continuous on $(0,\infty)$. If $K\in L^1$, then
\begin{displaymath} 
\lim_{\omega\to 0}\widehat{r}(\omega)=\frac{1}{\pi\int_0^\infty K(t)\d t}<\infty.	
\end{displaymath}
If $K$ is not in $L^1$, but satisfies \eqref{cond2b}, then Proposition \ref{prop:fourier.taube} implies $\lim_{\omega\to 0}\Kcos(\omega)=\lim_{\omega\to 0}\Ksin(\omega)=\infty$. It follows from \eqref{r.hat} that $\lim_{\omega\to 0}\widehat{r}(\omega)=0$. We see that in both cases, $\widehat{r}$ is  locally integrable around zero. Now as $\omega$ tends to infinity, by Lemma \ref{lem:fourier}, the numerator tends to 0 whereas the denominator is approximately $m^2\omega^2$, which implies that $\widehat{r}$ is integrable at infinity. We therefore conclude that $\widehat{r}$ belongs to $L^1(\rbb)$.
\end{proof}

Lemma \ref{lemma:r-hat-L1} implies that $\nu(\d\omega)=\widehat{r}(\omega)\d\omega$ satisfies \eqref{ineq:spectral-measure} with $k=0$. In view of Lemma \ref{lem:V-op-defined}, $\widehat{r}$ is the spectral density of some operator $V$ defined as in \eqref{eqn:V.op}. The following Theorem asserts that $V$ is indeed the weak solution of \eqref{eq:gle-weak}.

\begin{theorem} \label{thm5} Suppose that $m>0$ and $\lambda=0$ in \eqref{eq:gle-weak}. Let $K(t)$ satisfy Assumption \ref{a:K}. Then $V$ is a weak solution for \eqref{eq:gle-weak} if and only if the spectral measure $\nu$ satisfies $\nu(\d\omega)=\widehat{r}(\omega)\d\omega$ where $\widehat{r}$ is defined as in \eqref{eqn:spectral-density}.
\end{theorem}
\begin{proof} ($\Rightarrow$) Suppose $V$ is a weak solution for \eqref{eq:gle-weak}. By Proposition \ref{prop:Fourier.S'}\eqref{prop:Fourier.S':b},
\begin{displaymath}
\F{-m\varphi'+\widetilde{K^+*\widetilde{\varphi}}}=\F{-m\varphi'}+\overline{\F{K^+*\widetilde{\varphi}} } =\overline{im\omega}\widehat{\varphi}+ \overline{\widehat{K^+}\overline{\widehat{\varphi}}}=\overline{im\omega+\widehat{K^+}}\widehat{\varphi}.	
\end{displaymath}
We thus have that
\begin{align*}
\MoveEqLeft[5] \E{ \langle V,-m\varphi'+\widetilde{K^+*\widetilde{\varphi}}\rangle \overline{ \langle V,-m\psi'+\widetilde{K^+*\widetilde{\psi}}\rangle}}\\
&= \int_\rbb \left(im\omega+\widehat{K^+}(\omega)\right)\overline{\widehat{\varphi}}(\omega)\overline{im\omega+\widehat{K^+}(\omega)}\widehat{\psi}(\omega)\nu(\d\omega)\\
&=\int_\rbb \overline{\widehat{\varphi}}(\omega)\widehat{\psi}(\omega)\left|im\omega+\widehat{K^+}(\omega)\right|^2\nu(\d\omega).
\end{align*}
One the other hand, by Proposition \ref{prop:Fourier.S'}\eqref{prop:Fourier.S':a},
\begin{equation*}
\begin{aligned}
\int_\rbb K(t)\left( \varphi*\widetilde{\psi} \right)(t)\d t=\int_\rbb \widehat{K}(\omega)\frac{\overline{\widehat{\varphi}}(\omega)\widehat{\psi}(\omega)}{2\pi}\d\omega.
\end{aligned}
\end{equation*}
Since $V$ is a weak solution, we obtain
\begin{displaymath}
\int_\rbb \overline{\widehat{\varphi}}(\omega)\widehat{\psi}(\omega)\left|im\omega+\widehat{K^+}(\omega)\right|^2\nu(\d\omega)=\int_\rbb \widehat{K}(\omega)\frac{\overline{\widehat{\varphi}}(\omega)\widehat{\psi}(\omega)}{2\pi}\d\omega.
\end{displaymath}
Since all functions in $\Sc$ are the Fourier transform of some other Schwartz functions, we can rewrite the above formula as 
\begin{displaymath}
\int_\rbb \varphi(\omega)\psi(\omega)\left|im\omega+\widehat{K^+}(\omega)\right|^2\nu(\d\omega)=\int_\rbb \widehat{K}(\omega)\frac{\varphi(\omega)\psi(\omega)}{2\pi}\d\omega.
\end{displaymath}

Now we can choose $\left\{\varphi_k\right\}_{k\geq 1} \subset \Sc$, $\left\{\psi_k\right\}_{k\geq 1} \subset \Sc$ to be non-negative and respectively increasing up to $1_{[a,b]}$ and $1$. The Monotone Convergence Theorem then implies
\begin{displaymath}
\int_a^b\left|im\omega+\widehat{K^+}(\omega)\right|^2\nu(\d\omega)=\int_a^b \frac{\widehat{K}(\omega)}{2\pi}\d\omega.
\end{displaymath}
Since the equation above holds for any $-\infty< a<b<\infty$, we conclude that $\nu$ admits the Radon-Nykodim derivative $\nu(\d\omega)=\widehat{r}(\omega)\d\omega$.

($\Leftarrow$) Suppose $\nu(\d\omega)=\widehat{r}(\omega)$. To check the first condition of Definition \ref{def:weak-solution}, in view of Proposition \ref{prop:Fourier.S'}\eqref{prop:Fourier.S':b}, it suffices to show that 
\begin{displaymath}
\int_\rbb\left|\F{K^+}(\omega)\widehat{\phi}(\omega)\right|^2\widehat{r}(\omega)\d\omega=\int_\rbb\left|\widehat{K^+}(\omega)\widehat{\phi}(\omega)\right|^2\widehat{r}(\omega)\d\omega<\infty.
\end{displaymath}
If $K\in L^1$, the inequality above is evident since $\left|\widehat{K^+}(\omega)\widehat{\phi}(\omega)
\right|^2\leq \|\widehat{K^+}\|^2_{L^\infty}\|\widehat{\phi}\|^2_{L^\infty}$ and $\widehat{r}\in L^1(\rbb)$ by Lemma \ref{lemma:r-hat-L1}.

If $K$ satisfies \eqref{cond2b}, as $\omega$ tends to infinity, Lemma \ref{lem:fourier} implies that $\left|\widehat{K^+}(\omega)\widehat{\phi}(\omega)\right|^2\rightarrow 0$. It follows that $\left|\widehat{K^+}(\omega)\widehat{\phi}(\omega)\right|^2\widehat{r}(\omega)$ is dominated for sufficiently large $\omega$ by $\widehat{r}$ which is integrable. On the other hand, to control the integrand near zero, notice that
\begin{align*}
\MoveEqLeft[5] 2\pi\left|\widehat{K^+}(\omega)\widehat{\phi}(\omega)\right|^2\widehat{r}(\omega) =\left|\widehat{\phi}(\omega)\right|^2 \frac{2\Kcos(\omega)\left(\left[\Kcos(\omega) \right]^2
+\left[\Ksin(\omega) \right]^2\right)}
{\left[\Kcos(\omega) \right]^2
+\left[m\omega- \Ksin(\omega) \right]^2}\\
& = \frac{\left|\widehat{\phi}(\omega)\right|^2}{\omega^{1-\alpha}}\times \frac{2\omega^{1-\alpha}\Kcos(\omega)\left(\left[\omega^{1-\alpha}\Kcos(\omega) \right]^2
+\left[\omega^{1-\alpha}\Ksin(\omega) \right]^2\right)}
{\left[\omega^{1-\alpha}\Kcos(\omega) \right]^2
+\left[m\omega^{2-\alpha}- \omega^{1-\alpha}\Ksin(\omega) \right]^2}.
\end{align*}
By Proposition \ref{prop:fourier.taube}, $\Kcos(\omega)$ and $\Ksin(\omega)$ can be controlled near the origin by $1/\omega^{1-\alpha}$. It follows that $\left|\widehat{K^+}(\omega)\widehat{\phi}(\omega)\right|^2\widehat{r}(\omega)$ is dominated by $\frac{\left|\widehat{\phi}(\omega)\right|^2}{\omega^{1-\alpha}}$ when $\omega$ is near zero. We conclude that $\widehat{K^+}\widehat{\phi}$ belongs to $L^2(\widehat{r})$. We thus have that
\begin{equation*}
\begin{aligned}
\MoveEqLeft[5] \E{ \langle V,-m\varphi'+\widetilde{K^+*\widetilde{\varphi}}\rangle \overline{ \langle V,-m\psi'+\widetilde{K^+*\widetilde{\psi}}\rangle}}\\
&= \int_\rbb \left(im\omega+\widehat{K^+}(\omega)\right)\overline{\widehat{\varphi}}(\omega)\overline{im\omega+\widehat{K^+}(\omega)}\widehat{\psi}(\omega)\widehat{r}(\omega)\d\omega\\
&=\int_\rbb \widehat{K}(\omega)\frac{\overline{\widehat{\varphi}}(\omega)\widehat{\psi}(\omega)}{2\pi}\d\omega\\
&=\int_\rbb K(t)\left( \varphi*\widetilde{\psi} \right)(t)\d t.
\end{aligned}
\end{equation*}
The proof is thus complete.
\end{proof}

\subsection{Weak solutions when $m > 0$ and $\lambda > 0$}
\label{sec:weak-solutions:mpos-lambdapos}

Similar to previous subsection, we introduce the following function $\rhat$.
\begin{lemma}\label{lem:r-hat-L1:mpos-lambdapos}
Let $\rhat$ be defined as
\begin{equation}\label{form:rhat-mpos-lambdapos}
\rhat(\omega)=\frac{2\lambda+\beta\widehat{K}(\omega)}{2\pi\left|im\omega+\lambda+\beta\widehat{K^+}(\omega)\right|^2}.
\end{equation} 
Suppose $K$ satisfies Assumption \ref{a:K}. Then $\rhat$ belongs to $L^1(\rbb)$.
\end{lemma}
\begin{proof}
Same as the proof of Lemma \ref{lemma:r-hat-L1}.
\end{proof}

Since $\rhat\in L^1(\rbb)$, it is the spectral density of some operator $V$ defined as in \eqref{eqn:V.op}. In current situation where $\lambda>0$, we will show that the weak solution $V$ of \eqref{eq:gle-weak} indeed admits $\rhat$ defined in \eqref{form:rhat-mpos-lambdapos} as the spectral density if we assume zero correlation between two stationary random distributions $\dot W$ and $F$. 
\begin{theorem}\label{thm:weak-solution-mpos-lambdapos} Suppose that $K$ satisfies Assumption \ref{a:K} and that $m>0$, $\lambda>0$. Let $V$ be a weak solution for \eqref{eq:gle-weak}. Then, the following statements are equivalent.
\begin{enumerate}[(a)]
\item The spectral measure $\nu$ admits the representation $\nu(\d\omega)=\rhat(\omega)\d\omega$ where $\rhat(\omega)$ is defined as in \eqref{form:rhat-mpos-lambdapos}.
\item For any $\varphi\in\Sc$, $\E{\langle W,\varphi\rangle\langle F,\varphi\rangle}=0$.
\end{enumerate}
\end{theorem}
\begin{proof} (a)$\Rightarrow$(b): On one hand, we have that 
\begin{align*}
\MoveEqLeft[8]\E{\langle V,-m\varphi'+\lambda\varphi+\beta\widetilde{K^+*\widetilde{\varphi}}\rangle\overline{\langle V,-m\psi'+\lambda\psi+\beta\widetilde{K^+*\widetilde{\psi}}\rangle}}\\&=\int_\rbb \left|im\omega+\lambda+\beta\widehat{K^+}(\omega)\right|^2\rhat(\omega)\overline{\widehat{\varphi}}\widehat{\psi}\d\omega\\
 &=\int_\rbb \left(2\lambda+\beta\widehat{K}(\omega)\right)\frac{\overline{\widehat{\varphi}}\widehat{\psi}}{2\pi}\d\omega\\
&= 2\lambda\int_\rbb \varphi(t)\psi(t)\d t+ \beta \int_\rbb K(t)\left( \varphi*\widetilde{\psi} \right)(t)\d t.
\end{align*}
On the other hand, 
\begin{multline*}
\E{\langle \sqrt{2\lambda} W+\sqrt{\beta}F,\varphi\rangle \overline{\langle \sqrt{2\lambda} W+\sqrt{\beta}F,\psi\rangle}}\\ = 2\lambda\int_\rbb \varphi(t)\psi(t)\d t+ \beta \int_\rbb K(t)\left( \varphi*\widetilde{\psi} \right)(t)\d t + \sqrt{2\lambda\beta}\E{\langle W,\varphi\rangle\overline{\langle F,\psi\rangle}+\overline{\langle W,\psi\rangle}\langle F,\varphi\rangle}.
\end{multline*}
Since $V$ is a weak solution, we obtain
$
\E{\langle W,\varphi\rangle\overline{\langle F,\psi\rangle}+\overline{\langle W,\psi\rangle}\langle F,\varphi\rangle}=0,
$
which is the same as
$
\E{\langle W,\varphi\rangle\langle F,\psi\rangle+\langle W,\psi\rangle\langle F,\varphi\rangle}=0,
$ because they are real random variables. Substituting $\psi$ with $\varphi$ now implies (b).

(b)$\Rightarrow$(a): substituting $\varphi$ with $\varphi+\psi$, (b) implies that
\[0=\E{\langle W,\varphi+\psi\rangle\langle F,\varphi+\psi\rangle}=\E{\langle W,\varphi\rangle\langle F,\psi\rangle+\langle W,\psi\rangle\langle F,\varphi\rangle}.\]
Reversing the order of the arguments above, we obtain
\[\int_\rbb \left|im\omega+\lambda+\beta\widehat{K^+}(\omega)\right|^2\overline{\widehat{\varphi}}\widehat{\psi}\nu(\d\omega)=\int_\rbb \left(2\lambda+\beta\widehat{K}(\omega)\right)\frac{\overline{\widehat{\varphi}}\widehat{\psi}}{2\pi}\d\omega.\]
Using approximating argument as in the proof of Theorem \ref{thm5}, we deduce that $\nu$ is absolutely continuous with respect to Lebesgue measure and that $\nu(\d\omega)=\rhat(\omega)\d\omega$. The proof is complete.
\end{proof}
\begin{rem}
Since $\lambda>0$, the Fourier cosine transform of $K$ need not be strictly positive.
\end{rem}

\subsection{Weak solutions when $m=0$ and $\lambda > 0$} 
\label{sec:weak-solutions:mzero-lambdapos}
In this case, the spectral density in formula \eqref{form:rhat-mpos-lambdapos} becomes
\begin{equation}\label{form:rhat-mzero-lambdapos}
\rhat(\omega)=\frac{2\lambda+\beta\widehat{K}(\omega)}{2\pi\left|\lambda+\beta\widehat{K^+}(\omega)\right|^2}.
\end{equation}
\begin{lemma} Let $K$ satisfy Assumption \ref{a:K}. Then, $\rhat$ defined as in \eqref{form:rhat-mzero-lambdapos} is the spectral density of a generalized operator $V$ defined as in Section \ref{sec:prelim:extension}.
\end{lemma}
\begin{proof}
We note that $\rhat$ is no longer integrable since $\lim_{\omega\to\infty}\rhat(\omega)\in(0,\infty)$. However, using the assumption that $\Kcos(\omega)\geq 0$, it follows from \eqref{form:rhat-mzero-lambdapos} that
\begin{equation}\label{ineq:rho-hat-bounded:mzero-lambdapos}
\rhat(\omega)=\frac{1}{2\pi}\times\frac{2\lambda+2\beta\Kcos(\omega)}{\left(\lambda+\beta\Kcos(\omega)\right)^2+\left(\beta\Ksin(\omega)\right)^2}\leq \frac{1}{\pi\lambda},\end{equation}
which implies that $\int_\rbb \frac{\rhat(\omega)\d\omega}{1+\omega^2}<\infty$. In other words, the measure $\nu(\d\omega)=\rhat(\omega)\d\omega$ satisfies \eqref{ineq:spectral-measure} with $k=1$. In view of Lemma \ref{lem:V-op-defined}, $\rhat$ is the spectral density of a generalized operator $V$ defined as in \eqref{eqn:V.op}.
\end{proof}

Similar to Theorem \ref{thm:weak-solution-mpos-lambdapos}, assuming zero correlation between $\dot W$ and $F$, we arrive at following Theorem.
\begin{theorem}\label{thm:weak-solution-mzero-lambdapos}  Suppose that $K$ satisfies Assumption \ref{a:K} and that $m=0$, $\lambda>0$. Let $V$ be a weak solution for \eqref{eq:gle-weak}. Then, the following statements are equivalent.
\begin{enumerate}[(a)]
\item The spectral measure $\nu$ admits the representation $\nu(\d\omega)=\rhat(\omega)\d\omega$ where $\rhat$ is defined as in \eqref{form:rhat-mzero-lambdapos}.
\item For any $\varphi\in\Sc$, $\E{\langle W,\varphi\rangle\langle F,\varphi\rangle}=0$.
\end{enumerate}
\end{theorem}
\begin{proof} Same as the proof of Theorem \ref{thm:weak-solution-mpos-lambdapos}.
\end{proof}

\subsection{Weak solutions when $m = 0$ and $\lambda = 0$}
\label{sec:weak-solutions:mzero-lambdazero} In this situation, the spectral density in formula \eqref{form:rhat-mpos-lambdapos} becomes
\begin{equation}\label{form:rhat-mzero-lambdazero}
\rhat(\omega)=\frac{\widehat{K}(\omega)}{\pi\beta\left|\widehat{K^+}(\omega)\right|^2}.
\end{equation}

Because the structure of $\rhat$ is quite different from previous three cases, we need to impose Assumption \ref{cond:cv} in addition to the Assumption \eqref{cond1} on the memory kernel $K(t)$. 
\begin{lemma} \label{lem:V-defined:mzero-lambdazero} Let $K(t)$ satisfy Assumption \ref{a:K} and Assumption \ref{cond:cv}. Then, $\rhat$ defined as in \eqref{form:rhat-mzero-lambdazero} is the spectral density of a generalized operator $V$ defined as in Section \eqref{sec:prelim:extension}.
\end{lemma}
\begin{proof} We need to check that $\nu(\d\omega)=\rhat(\omega)\d\omega$ in this case satisfies \eqref{ineq:spectral-measure}. Indeed, we claim that Inequality \eqref{ineq:spectral-measure} holds with $k=1$, namely
\begin{equation} \label{ineq:cv-gle7-1}
\int_0^\infty \frac{\pi\beta}{2}\rhat(\omega)\frac{1}{1+\omega^2}\d\omega=	\int_0^\infty \frac{\Kcos(\omega)}{\left(\Kcos(\omega)^2+\Ksin(\omega)^2\right)(1+\omega^2)}<\infty.
\end{equation}
When $\omega$ is near zero, we have that
\begin{equation} \label{ineq:cv-gle7-1a}
\frac{\pi\beta}{2}\rhat(\omega)=\frac{\Kcos(\omega)}{\Kcos(\omega)^2+\Ksin(\omega)^2}\leq \frac{1}{\Kcos(\omega)}\to\frac{1}{\Kcos(0^+)},
\end{equation}
which is either finite or zero depending on $K$ integrable or not, respectively. In other word, $\widehat{r}(\omega)$ is always bounded near the origin. The only concern now is when $\omega$ tends to infinity. Since $K$ satisfies Assumption \ref{cond:cv}, Proposition \ref{prop:fourier-cv} implies the existence of $\sigma\in (0,1)$ and $c(\sigma)>0$ such that for all $\omega$ sufficiently large
\begin{equation}\label{ineq:cv-gle7-2}
\frac{\Kcos(\omega)}{\Kcos(\omega)^2+\Ksin(\omega)^2}\leq c(\sigma)\omega^\sigma.
\end{equation}
To see that, suppose $K$ satisfies \eqref{cond:cv-finite}. Let $\sigma_1$ be the power constant from \eqref{cond:cv-finite}. We estimate
\begin{equation}\label{ineq:cv-gle7-3}
\frac{\Kcos(\omega)}{\Kcos(\omega)^2+\Ksin(\omega)^2}\leq \frac{\omega^{2-\sigma_1}\Kcos(\omega)}{\left[\omega\Ksin(\omega)\right]^2}\omega^{\sigma_1}.
\end{equation}
We invoke \eqref{lim:fourier-cv-1} to find
\begin{equation}\label{ineq:cv-gle7-4}
\lim_{\omega\to\infty} \frac{\omega^{2-\sigma_1}\Kcos(\omega)}{\left[\omega\Ksin(\omega)\right]^2}=0,
\end{equation}
and thus infer the constants $\sigma$ and $c(\sigma)$ in \eqref{ineq:cv-gle7-2}, say $\sigma=\sigma_1$ and $c(\sigma)=1$. On the other hand, suppose $K$ satisfies \eqref{cond:cv-infinite}. Let $\sigma_2$ be the power constant from \eqref{cond:cv-infinite}. Similar to \eqref{ineq:cv-gle7-3}, we estimate
\begin{equation}\label{ineq:cv-gle7-5}
\frac{\Kcos(\omega)}{\Kcos(\omega)^2+\Ksin(\omega)^2}\leq \frac{\omega^{1-\sigma_2}\Kcos(\omega)}{\left[\omega^{1-\sigma_2}\Ksin(\omega)\right]^2}\omega^{1-\sigma_2}.
\end{equation}
It follows from \eqref{lim:fourier-cv-2} that
\begin{equation}\label{ineq:cv-gle7-6}
\lim_{\omega\to\infty}\frac{\omega^{1-\sigma_2}\Kcos(\omega)}{\left[\omega^{1-\sigma_2}\Ksin(\omega)\right]^2} \in(0,\infty).
\end{equation}
Setting $\sigma=1-\sigma_2$, $c=2\lim_{\omega\to\infty}\frac{\omega^{1-\sigma_2}\Kcos(\omega)}{\left[\omega^{1-\sigma_2}\Ksin(\omega)\right]^2}$, we obtain \eqref{ineq:cv-gle7-2}. We conclude that $\nu$ satisfies condition \eqref{ineq:spectral-measure}, which completes the proof.
\end{proof}
Using the same Definition \ref{def:weak-solution} for weak solution with $m=\lambda=0$, we have the following Theorem. 
\begin{theorem} \label{thm:weak-solution:mzero-lambdazero} Suppose that $m=\lambda=0$. Let $K(t)$ satisfy Assumption \ref{a:K} and Assumption \ref{cond:cv}. Then $V$ is a weak solution for \eqref{eq:gle-weak} if and only if $\nu(\d\omega)=\rhat(\omega)\d\omega$ where $\rhat(\omega)$ is given by formula \eqref{form:rhat-mzero-lambdazero}.
\end{theorem}
\begin{proof} The proof is essentially the same as that of Theorem \ref{thm5}. 
\end{proof}

\section{Regularity}
\label{sec:regularity}
We organize this section the same as Section \ref{sec:weak-solutions}. We begin with the case $(m>0,\lambda=0)$. The other two cases $(m>0,\lambda>0)$ and $(m=0,\lambda>0)$ are handled using similar arguments. The last case $(m=0,\lambda=0)$ is treated differently. In addition, using classical theory of regularity of Gaussian processes, we will show that in the first case, with further assumptions on the memory kernel, $V(t)$ is differentiable almost surely. 
\subsection{Regularity when $m>0$ and $\lambda=0$}\label{sec:regularity:mpos-lambdazero}
We begin with the fact that the velocity $V(t)$ is well-defined as a stochastic process in time.
\begin{proposition}\label{prop:V(t)-mpos-lambdazero} Under the same hypotheses as Theorem \ref{thm5}, let $V$ be the weak solution of \eqref{eq:gle-weak}. Then the process $V(t)=\la V,\delta_t\ra$ is well-defined.
\end{proposition}
\begin{proof}
By Lemma \ref{lemma:r-hat-L1}, $\widehat{r}$ belongs to $L^1$. In view of Lemma \ref{lem:finite-nu}, the spectral measure $\nu(\d\omega)=\widehat{r}(\omega)\d\omega$ is finite, which implies that $V(t)=\langle V,\delta_t\rangle$ is indeed a stationary, mean-square continuous Gaussian process.
\end{proof}

In order to establish the regularity of a Gaussian process, we shall employ the following classic lemmas from Chapter 9.3, \cite{cramer1967stationary}. 
\begin{lemma}\label{lem1}
If a real stationary Gaussian process $\xi(t)$ with covariance function $k(t)=\int_\rbb e^{it\omega}\nu(\d\omega)$ satisfies 
\[\int_0^\infty \left[\log(1+\omega)\right]^a \nu(\d\omega)<\infty,\]
for some $a>3$, then $\xi(t)$ is equivalent to a process $\eta(t)$ which a.s. is continuous.
\end{lemma}
\begin{lemma}\label{lem2}
If a real stationary Gaussian process $\xi(t)$ with covariance function $k(t)=\int_\rbb e^{it\omega}\nu(\omega)$ satisfies 
\[\int_0^\infty \omega^2\left[\log(1+\omega)\right]^a \nu(\omega)<\infty,\]
for some $a>3$, then $\xi(t)$ is equivalent to a process $\eta(t)$ which is a.s.~continuously differentiable.
\end{lemma}

We are now ready to assert the regularity of $V(t)$.
\begin{theorem} \label{thm:V-Continuous}
Under the same hypotheses as Theorem \ref{thm5}, let $V(t)$ be the Gaussian process defined in Proposition \ref{prop:V(t)-mpos-lambdazero}. Then $V(t)$ is continuous.
\end{theorem}
\begin{proof}
The continuity of $V(t)$ will follow from Lemma \ref{lem1} if it holds that
\[\int_0^\infty \left[\log(1+\omega)\right]^a \widehat{r}(\omega) \d\omega <\infty, \]
where $a>3$ and $\rhat$ is defined as in \eqref{r.hat}. The only issue here is when $\omega$ tends to infinity. However, for any $a>3$, we note that
\begin{displaymath}
2\pi\left[\log(1+\omega)\right]^a\widehat{r}(\omega) = \frac{2\left[\log(1+\omega)\right]^a\Kcos(\omega)}
{\left[\Kcos(\omega) \right]^2
+\omega^2\left[m- \frac{1}{\omega}\Ksin(\omega) \right]^2}.
\end{displaymath}
In view of Lemma \ref{lem:fourier}, $\lim_{\omega\to\infty}\Kcos(\omega)=
\lim_{\omega\to\infty}\Ksin(\omega) =0$. Hence, when $\omega$ is large, $\left[\log(1+\omega)\right]^a\widehat{r}(\omega)$ is dominated by $\left[\log(1+\omega)\right]^a/\omega^2$ which is integrable. We therefore conclude that $\left[\log(1+\omega)\right]^a \widehat{r}(\omega)\in L^1[0,\infty).$
\end{proof}

As a consequence of $V(t)$ being continuous, we immediately obtain the following.

\begin{corollary} \label{cor:X(t)-diff} Under the same hypotheses as Theorem \ref{thm5}, $X(t)$ is a.s.~differentiable where $X(t)=\int_0^t V(s)\d s$.
\end{corollary}

We finally assert the differentiablity of $V(t)$. 
\begin{theorem} \label{thm:V-differentiable}
Under the same hypotheses as Theorem \ref{thm5}, let $V(t)$ be as in Proposition \ref{prop:V(t)-mpos-lambdazero}. Assume further that $K$ is positive definite and that for some $b>3$
\begin{equation} \label{cond:differentiability}
K(0)-K(t)=O\left(\left|\log t \right|^{-b}\right),\ t\to 0^+.
\end{equation}
Then the Gaussian process $V(t)$ is a.s.~continuously differentiable.
\end{theorem}
\begin{proof} By Proposition \ref{prop:Fourier-Inverse}, $\widehat{K}$ is integrable and $K$ admits the inverse formula
\[K(t)=\frac{1}{2\pi}\int_\rbb e^{it\omega}\widehat{K}(\omega)\d\omega=\int_0^\infty\cos(t\omega)\frac{\widehat{K}(\omega)}{\pi}\d\omega.\]
In view of Lemma 2 from Section 9.3, \cite{cramer1967stationary}, we deduce that for any $a<b$
\begin{equation}\label{ineq:V-diff-1}
\int_0^\infty |\log(1+\omega)|^a\widehat{K}(\omega)\d\omega<\infty.
\end{equation}
By Proposition \ref{prop:Fourier.S'}\eqref{prop:Fourier.S':a}, the above inequality is equivalent to
\begin{equation}
\label{ineq:V-diff-2}\int_0^\infty |\log(1+\omega)|^a\Kcos(\omega)\d\omega<\infty.
\end{equation}
Now the differentiability of $V(t)$ follows immediately from Lemma \ref{lem2} if we can show
\[\int_0^\infty \omega^2\left[\log(1+\omega)\right]^a \widehat{r}(\omega) \d\omega <\infty, \]
which is the same as
\begin{equation} \label{ineq:V-diff-4}
\int_0^\infty  \frac{\omega^2\left[\log(1+\omega)\right]^a\Kcos(\omega) }
{\left[\Kcos(\omega) \right]^2
+\omega^2\left[m- \frac{1}{\omega}\Ksin(\omega) \right]^2} \d\omega <\infty. 
\end{equation}
On one hand, when $\omega$ is near the origin, the integrand in \eqref{ineq:V-diff-4} is dominated by $\rhat$ which is integrable by virtue of Lemma \ref{lemma:r-hat-L1}. On the other hand, when $\omega$ becomes large, reasoning as in the proof of Theorem \ref{thm5}, we see that, the integrand is dominated by $\left[\log(1+\omega)\right]^a\Kcos(\omega)$, which is also integrable thanks to \eqref{ineq:V-diff-2}. We therefore obtain \eqref{ineq:V-diff-4} which in turns implies the differentiability of $V(t)$. The proof is complete.
\end{proof}

\subsection{Regularity when $m>0$ and $\lambda > 0$} \label{sec:regularity:mpos-lambdapos}
\begin{proposition} \label{prop:V(t)-defined:mpos-lambdapos} Under the same Hypothesis of Theorem \ref{thm:weak-solution-mpos-lambdapos}, let $V$ be the weak solution of \eqref{eq:gle-weak}. Then, the velocity process $V(t)=\langle V,\delta_t\rangle$ is well-defined.
\end{proposition}
\begin{proof} In view of Lemma \ref{lem:finite-nu}, we need to check that the spectral measure $\nu(\d\omega)=\widehat{r}(\omega)\d\omega$ is finite, where $\widehat{r}$ is defined in \eqref{form:rhat-mpos-lambdapos}. This in turns follows immediately from Lemma \ref{lem:r-hat-L1:mpos-lambdapos}.
\end{proof}

We assert that $V(t)$ is always continuous in this case.
\begin{theorem} \label{thm:V(t)-Continuous:mpos-lambdapos} Under the same Hypothesis of Theorem \ref{thm:weak-solution-mpos-lambdapos}, let $V(t)$ be the Gaussian process from Proposition \ref{prop:V(t)-defined:mpos-lambdapos}. Then $V(t)$ is continuous.
\end{theorem}
\begin{proof}
The proof is similar to that of Theorem \ref{thm:V-Continuous}.
\end{proof}

We immediately obtain the differentiability of the particle position process $X(t)$.
\begin{corollary} Under the same Hypothesis of Theorem \ref{thm:weak-solution-mpos-lambdapos}, $X(t)$ is a.s. differentiable where $X(t)=\int_0^t V(s)\d s$.
\end{corollary}
%

\subsection{Regularity when $m=0$ and $\lambda > 0$} \label{sec:regularity:mzero-lambdapos}
\begin{proposition} \label{prop:V(t)-defined:mzero-lambdapos} Under the same Hypothesis of Theorem \ref{thm:weak-solution-mzero-lambdapos}, let $V$ be the weak solution of \eqref{eq:gle-weak}. Then, the velocity process $V(t)=\langle V,\delta_t\rangle$ is not well-defined, but the particle position process $X(t)=\langle V,1_{[0,t]}\rangle$ is.
\end{proposition}
\begin{proof} We recall that the spectral density $\rhat(\omega)$ from \eqref{form:rhat-mzero-lambdapos} satisfies $\lim_{\omega\to\infty}\rhat(\omega)\in(0,\infty)$. This implies that $\rhat\notin L^1(\rbb)$. In view of Lemma \ref{lem:finite-nu}, $V(t)$ is not well-defined since the spectral measure $\nu(\d\omega)=\widehat{r}(\omega)\d\omega$ is not finite. However, $X(t)=\langle V,1_{[0,t]}\rangle$ is well-defined since $1_{[0,t]}\in L^2(\rhat)$. To see that, we invoke Inequality \eqref{ineq:rho-hat-bounded:mzero-lambdapos} to estimate
\begin{displaymath}
\int_\rbb \left|\widehat{1_{[0,t]}}(\omega)\right|^2\rhat(\omega)\d\omega=2\int_\rbb \frac{1-\cos(t\omega)}{\omega^2}\rhat(\omega)\d\omega<\frac{2}{\pi\lambda}\int_\rbb \frac{1-\cos(t\omega)}{\omega^2}\d\omega<\infty.
\end{displaymath}
\end{proof}

Since $V(t)$ is not well-defined, it is not certain if $X(t)$ is differentiable. We however are able to assert the continuity of $X(t)$.
\begin{theorem} \label{thm:X(t)-continuous:mzero-lambdapos} Under the same hypotheses as Theorem \ref{thm:weak-solution-mzero-lambdapos}, $X(t)$ is a.s.~continuous.
\end{theorem}
\begin{proof}
In view of Proposition 3.18, \cite{hairer2009introduction}, it suffices to show that for fixed $T$, there exists $\kappa>0$ s.t. for all $0\leq s<t\leq T$, 
\begin{equation} \label{ineq:X(t)-Kolmogorov}
\Enone{\left|X(t)-X(s)\right|^2}\leq c_\kappa|t-s|^\kappa,
\end{equation}
where $c_\kappa>0$ is a constant. A straightforward calculation yields
\begin{equation} \label{ineq:X(t)-Kolmogorov-1}
\Enone{\left|X(t)-X(s)\right|^2}=\int_\rbb \left|\widehat{1_{[s,t]}}(\omega)\right|^2\rhat(\omega)\d\omega=4\int_0^\infty \frac{1-\cos\left((t-s)\omega\right)}{\omega^2}\rhat(\omega)\d\omega.
\end{equation}
Here we shall employ two elementary inequalities: for all $x\in\rbb$,
\begin{equation}\label{ineq:x-cont-1a}
1-\cos(x)\leq \frac{x^2}{2},
\end{equation}
 and that for every $\eta\in(0,1)$, there exists $c_
\eta>0$ such that for all $x$,
\begin{equation} \label{ineq:x-cont-1}
1-\cos(x)\leq c_\eta x^\eta.
\end{equation}
We estimate the last term of \eqref{ineq:X(t)-Kolmogorov-1} using \eqref{ineq:x-cont-1} with $\eta=1/2$,
\begin{align*}
\MoveEqLeft[5]\int_0^\infty \frac{1-\cos\left((t-s)\omega\right)}{\omega^2}\rhat(\omega)\d\omega \\&= \int_0^1 \frac{1-\cos\left((t-s)\omega\right)}{\omega^2}\rhat(\omega)\d\omega+\int_1^\infty \frac{1-\cos\left((t-s)\omega\right)}{\omega^2}\rhat(\omega)\d\omega\\
&\leq |t-s|^2\int_0^1\rhat(\omega)\d\omega+c_{1/2}|t-s|^{1/2}\int_1^\infty \frac{1}{\omega^{3/2}}\rhat(\omega)\d\omega,
\end{align*}
where in the last implication, we use \eqref{ineq:x-cont-1a} on the first term and \eqref{ineq:x-cont-1} with $\eta=1/2$ on the second term. We finally recall the fact that $\rhat$ is bounded by $1/{\pi\lambda}$ from Inequality \eqref{ineq:rho-hat-bounded:mzero-lambdapos} to obtain \eqref{ineq:X(t)-Kolmogorov} with $\kappa=1/2$. The proof is thus complete.
\end{proof}

\subsection{Regularity when $m = 0$ and $\lambda = 0$} In this situation, once again $V(t)$ is not well-defined but $X(t)$ is. We therefore are only able to investigate the continuity of $X(t)$. We begin by the following proposition.
\begin{proposition}\label{prop:V(t)-defined:mzero-lambdazero} Under the same hypotheses as Theorem \ref{thm:weak-solution:mzero-lambdazero}, let $V$ be the weak solution of \eqref{eq:gle-weak}. Then, $V(t)=\langle V,\delta_t\rangle$ is not well-defined, but $X(t)=\langle V,1_{[0,t]}\rangle$ is.
\end{proposition}
\begin{proof} (a) $V(t)$ is not-well-defined: In view of Lemma \ref{lem:finite-nu}, it suffices to show that $\rhat$ from \eqref{form:rhat-mzero-lambdazero} is not intergable, which implies that $\nu(\d\omega)=\widehat{r}(\omega)\d\omega$ is infinite. There are two cases: 

If $K$ satisfies \eqref{cond:cv-finite}, we write $\rhat$ as 
\begin{equation} \label{eqn:cv-gle7-7a}
\frac{\pi\beta}{2}\rhat(\omega)=\frac{\Kcos(\omega)}{\Kcos(\omega)^2+\Ksin(\omega)^2}=\frac{\omega^{2}\Kcos(\omega)}{\left[\omega\Kcos(\omega)\right]^2+\left[\omega\Ksin(\omega)\right]^2}.
\end{equation}
It follows from \eqref{lim:fourier-cv-1} that $\lim_{\omega\to\infty}\left[\omega\Kcos(\omega)\right]^2+\left[\omega\Ksin(\omega)\right]^2=K(0)^2$. It remains to show that $\omega^2\Kcos(\omega)$ is not integrable at infinity. Indeed, we recall from Lemma \ref{lem4} that for all non-zero $\omega$, $\omega^2\Kcos(\omega)=\int_0^\infty K''(t)\left(1-\cos(t\omega)\right)\d t$. Since for all $t$, $K''(t)$ is not identical to zero  and $K''(t)$ is continuous, we assume that there exists an interval $(\epsilon_1,\epsilon_2)$ such that $K''(t)>0$ on for $t\in(\epsilon_1,\epsilon_2)$. We now integrate with respect to $\omega$ to find
\begin{displaymath}
	\begin{aligned}
    \int_A^\infty \!\! \omega^2\Kcos(\omega)\d\omega &= \int_A^\infty \int_0^\infty \!\!K''(t)\left(1-\cos(t\omega)\right)\d t\d\omega\\ 
	& \geq \int_{\epsilon_1}^{\epsilon_2} K''(t)\int_A^\infty \!\! \left(1-\cos(t\omega)\right)\d\omega\d t =\infty,	
	\end{aligned}
\end{displaymath}
since for all $t\in (\epsilon_1,\epsilon_2)$, it is clear that $\int_A^\infty\left(1-\cos(t\omega)\right)\d\omega=\infty$.

If $K$ satisfies \eqref{cond:cv-infinite}, we observe that
\begin{equation} \label{eqn:cv-gle7-7}
\frac{\pi\beta}{2}\rhat(\omega)=\frac{\Kcos(\omega)}{\Kcos(\omega)^2+\Ksin(\omega)^2}=\frac{\omega^{1-\sigma_2}\Kcos(\omega)}{\left[\omega^{1-\sigma_2}\Kcos(\omega)\right]^2+\left[\omega^{1-\sigma_2}\Ksin(\omega)\right]^2}\omega^{1-\sigma_2},
\end{equation}
where $\sigma_2$ is the constant from \eqref{cond:cv-infinite}. We invoke \eqref{lim:fourier-cv-2} to find that $\rhat(\omega)\sim \omega^{1-\sigma_2}$ as $\omega\to\infty$. 

We therefore conclude from both cases that $\rhat(\omega)\notin L^1$.

(b) $X(t)$ is well-defined: This will follow immediately from Definition \ref{defn:form:V(t)} if we can show that
	 \begin{equation}\label{ineq:cv-gle7-8}
\int_\rbb \left|\F{1_{[0,t]}}(\omega)\right|^2\rhat(\omega)\d\omega<\infty,
\end{equation}
which is equivalent to
\begin{equation}\label{ineq:cv-gle7-9}
\int_\rbb \frac{1-\cos(t\omega)}{\omega^2}\times\frac{\Kcos(\omega)}{\Kcos(\omega)^2+\Ksin(\omega)^2}\d\omega<\infty,
	 \end{equation}
since $\F{1_{[0,t]}}(\omega)=\widehat{1_{[0,t]}}(\omega)=\frac{1-e^{-it\omega}}{i\omega}$. When $\omega$ is near the origin, we recall from \eqref{ineq:cv-gle7-1a} that $\rhat$ is always bounded regardless of the integrability of $K(t)$. The only concern is when $\omega$ tends to infinity. To this end, we employ \eqref{ineq:cv-gle7-2} to infer for all $\omega$ sufficiently large
\begin{equation}\label{ineq:cv-gle7-10}
 \frac{1-\cos(t\omega)}{\omega^2}\times\frac{\Kcos(\omega)}{\Kcos(\omega)^2+\Ksin(\omega)^2}< c(\sigma)\frac{1-\cos(t\omega)}{\omega^{2-\sigma}}.
\end{equation}
The RHS above is clearly integrable near infinity. We hence obtain \eqref{ineq:cv-gle7-8}.
\end{proof}
\begin{theorem} \label{thm:X(t)-continuous:mzero-lambdazero}  Under the same hypotheses as Theorem \ref{thm:weak-solution:mzero-lambdazero}, let $X(t)$ be the particle position process from Proposition \ref{prop:V(t)-defined:mzero-lambdazero}. Then, $X(t)$ is continuous a.s.
\end{theorem}
\begin{proof} To show continuity, we apply a similar argument as in the proof of Theorem \ref{thm:X(t)-continuous:mzero-lambdapos}. We recall from \eqref{ineq:x-cont-1} that for the constant $\sigma$ in \eqref{ineq:cv-gle7-2}, there exists $c_\sigma>0$ such that for all $x\in\rbb$,
\begin{equation}\label{ineq:cv-gle7-11}
1-\cos(x)\leq c_\sigma x^{(1-\sigma)/2}.
\end{equation}
and that $1-\cos(x)\leq x^2/2$. We now fix $A$ large enough such that for $\omega \geq A$, \eqref{ineq:cv-gle7-2} holds. We then estimate for $t\neq s$ arbitrarily given,
\begin{equation}\label{ineq:cv-gle7-12}
\begin{aligned}
\MoveEqLeft[5]\int_0^\infty \frac{1-\cos\left((t-s)\omega\right)}{\omega^2}\rhat(\omega)\d\omega \\&= \int_0^A \frac{1-\cos\left((t-s)\omega\right)}{\omega^2}\rhat(\omega)\d\omega+\int_A^\infty \frac{1-\cos\left((t-s)\omega\right)}{\omega^2}\rhat(\omega)\d\omega\\
&\leq |t-s|^2\int_0^A\rhat(\omega)\d\omega+c_\sigma|t-s|^{(1-\sigma)/2}\int_1^\infty \frac{c}{\omega^{1+(1-\sigma)/2}}\d\omega,
\end{aligned}
\end{equation}
where $c$, $c_\sigma$ are from \eqref{ineq:cv-gle7-2}, \eqref{ineq:cv-gle7-11} respectively. We thus obtain an estimate similar to \eqref{ineq:X(t)-Kolmogorov}, namely, for $0\leq s,t\leq T$, there exixts $C=C(T)$ such that
\begin{equation} \label{ineq:cv-gle7-13}
\Enone{\left|X(t)-X(s)\right|^2}\leq C|t-s|^{(1-\sigma)/2}.
\end{equation}
The continuity of $X(t)$ follows immediately from Proposition 3.18, \cite{hairer2009introduction}, which concludes the proof.
\end{proof}

\section{Asymptotic analysis of the Mean-Squared Displacement}
\label{sec:msd}

We are now prepared to prove our version of the Meta-Theorem \eqref{eq:meta-thm} that was presented in the introduction. Having established basic properties of the spectral density $\hat{r}$ in the last two sections,  the Abelian Theorem for Fourier Transforms from Section \ref{sec:tauberian} will allow us to immediately handle the case when $m > 0$ or $\lambda > 0$. As has been the case throughout the paper, $m = \lambda = 0$ presents a greater challenge and requires more restrictions on the memory $K(t)$.

Throughout this section, let $X(t)$ be the GLE position process as defined by Definition \ref{defn:form:V(t)}.

\subsection{Asymptotic of the MSD when either $m>0$ or $\lambda>0$}
\begin{theorem} \label{thm:diffusion}
Suppose that either $m > 0$ or $\lambda > 0$ and assume $K$ satisfies \eqref{cond1}$+$\eqref{cond2a}. Then
\[\lim_{t\to\infty}\frac{\E{X^2(t)}}{t} = C\in (0,\infty), \]
i.e. the process $X(t)$ is asymptotically diffusive.
\end{theorem}

\begin{proof} Using Definition \eqref{form:V(t)}, we have
\begin{equation} \label{eqn:X(t)}
\E{X^2(t)}=\int_\rbb \left|\widehat{1_{[0,t]}}(\omega)\right|^2\widehat{r}(\omega)\d\omega\\
= 2\int_\rbb \frac{1-\cos(t\omega)}{\omega^2} \widehat{r}(\omega)\d\omega.
\end{equation}
By changing variable $z:=t\omega$, we obtain
\[\E{X^2(t)}=2t\int_\rbb \frac{1-\cos(z)}{z^2}
\widehat{r}\left(\frac{z}{t}\right)\d z,\]
which implies
\begin{eqnarray} \label{displace1}
\frac{\E{X^2(t)}}{t}&=&2\int_\rbb \frac{1-\cos(z)}{z^2}
\widehat{r}\left(\frac{z}{t}\right)\d z.
\end{eqnarray}
We remind the reader that, by Equation \eqref{form:rhat-mpos-lambdapos}, the general form of $\widehat{r}(\omega)$ is 
\begin{displaymath}
\widehat{r}(\omega) = \frac{2\lambda+\beta\widehat{K}(\omega)}{2 \pi \big|im\omega+\lambda+\beta\widehat{K^+}(\omega)\big|^2}.
\end{displaymath}
Since $K$ is integrable by Condition \eqref{cond2a}, either $m>0$ or $\lambda>0$ implies that $\lim_{\omega\to 0}\widehat{r}(\omega)=\rhat(0)\in(0,\infty)$. In addition, by Condition \eqref{cond1}, Lemma \ref{lem1} implies that $\rhat(\omega)$ is bounded at infinity. As a consequence, $\rhat(\omega)$ is bounded in $\rbb$. By the Dominated Convergence Theorem, we obtain
\begin{displaymath}
\lim_{t\to\infty}\frac{\E{X^2(t)}}{t} = 2\int_\rbb \frac{1-\cos(z)}{z^2}
\widehat{r}(0)\d z\in (0,\infty).
\end{displaymath}

\end{proof}

\begin{theorem} \label{thm:subdiffusion-gle}
Suppose that either $m > 0$ or $\lambda > 0$ and assume $K$ satisfies~\eqref{cond1}~$+$~\eqref{cond2b}. Then
\[\lim_{t\to\infty}\frac{\E{X^2(t)}}{t^{\alpha}}=C\in (0,\infty),\]
where $\alpha$ is the constant from condition \eqref{cond2b}.
\end{theorem}
\begin{proof}
From \eqref{displace1}, we have
\begin{equation*}
\frac{\E{X^2(t)}}{t^{\alpha}}=2\int_\rbb \frac{1-\cos(z)}{z^{1+\alpha}}\times
\frac{\widehat{r}\left(\frac{z}{t}\right)}{\left(\frac{z}{t}\right)^{1-\alpha}}\d z.
\end{equation*}
We observe that \eqref{form:rhat-mpos-lambdapos} is equivalent to
\begin{displaymath}
\frac{\widehat{r}(\omega)}{\omega^{1-\alpha}} =\frac{1}{\pi} \times\frac{\lambda\omega^{1-\alpha}+\beta\omega^{1-\alpha}\Kcos(\omega)}
{\left[\lambda\omega^{1-\alpha}+\beta\omega^{1-\alpha}\Kcos(\omega) \right]^2
+\left[m\omega^{2-\alpha}- \beta\omega^{1-\alpha}\Ksin(\omega) \right]^2}.
\end{displaymath}
Proposition \ref{prop:fourier.taube} implies that
\begin{displaymath}
\lim_{\omega\to 0}\frac{\widehat{r}(\omega)}{\omega^{1-\alpha}}=c\in(0,\infty),
\end{displaymath}
and subsequently, $\widehat{r}(\omega)/\omega^{1-\alpha}$ is bounded on $(0,\infty)$ since by Lemma \ref{lem1}, $\rhat(\omega)$ is bounded at infinity. Applying the Dominated Convergence Theorem gives
\begin{equation*}
\lim_{t\to\infty}\frac{\E{X^2(t)}}{t^{\alpha}}=2c\int_\rbb \frac{1-\cos(z)}{z^{1+\alpha}}\d z\in(0,\infty).
\end{equation*}
The proof is complete.
\end{proof}

\subsection{Asymptotics of the MSD when $m = \lambda = 0$} \label{sec:msd:mzero-lambdazero}

\begin{theorem}\label{thm:X(t)-diffusive:mzero-lambdazero} Suppose that $m=\lambda=0$ and that $K$ satisfies Assumption~\ref{cond:cv}~$+$~ \eqref{cond1}. Then, \begin{enumerate}[(a)]
		\item If $K$ satisfies \eqref{cond2a}, 
		\[\lim_{t\to\infty}\frac{\E{X^2(t)}}{t} = C\in (0,\infty). \]
		\item If $K$ satisfies \eqref{cond2b}, 
		\[\lim_{t\to\infty}\frac{\E{X^2(t)}}{t^{\alpha}}=C\in (0,\infty),\]
		where $\alpha$ is the constant from condition \eqref{cond2b}.
	\end{enumerate}
\end{theorem}
\begin{proof} (a) We recall from \eqref{displace1} that
	\begin{equation}\label{eqn:cv-gle7-13}
\frac{\E{X^2(t)}}{t}=2\int_\rbb \frac{1-\cos(z)}{z^2}
\rhat\left(\frac{z}{t}\right)\d z.
	\end{equation}
It therefore suffices to show that
\begin{equation}\label{eqn:cv-gle7-14}
\lim_{t\to\infty}\int_0^\infty \frac{1-\cos(z)}{z^2}
\rhat\left(\frac{z}{t}\right)\d z=C\in(0,\infty).
\end{equation}
Fixing $A$ such that for all $\omega\geq A$, \eqref{ineq:cv-gle7-2} holds, we decompose the integral above as
\begin{equation}\label{eqn:cv-gle7-15}
\int_0^\infty \frac{1-\cos(z)}{z^2}
\rhat\left(\frac{z}{t}\right)\d z=\int_0^{At}+\int_{At}^\infty \frac{1-\cos(z)}{z^2}
\rhat\left(\frac{z}{t}\right)\d z=I_5(t)+I_6(t).
\end{equation}
We claim that $I_6(t)\to 0$ as $t\to\infty$. Indeed, since $z\geq At$, by \eqref{ineq:cv-gle7-2}, $\rhat(z/t)\leq c(z/t)^\sigma$. We have a chain of implications.
\begin{equation}\label{eqn:cv-gle7-16}
I_6(t)\leq \int_{At}^\infty \frac{1-\cos(z)}{z^2}
\times c\left(\frac{z}{t}
\right)^{\sigma}\d z= \frac{c}{t^\sigma}\int_{At}^\infty \frac{1-\cos(z)}{z^{2-\sigma}}\to 0.
\end{equation}
We write $I_5(t)$ as
\begin{equation}\label{eqn:cv-gle7-17}
I_5(t)=\int_0^{\infty} 1_{(0,At]}(z) \frac{1-\cos(z)}{z^2}
\rhat\left(\frac{z}{t}\right)\d z.
\end{equation}
Since $\rhat(\omega)$ is bounded on $(0,A]$ and $z/t\leq A$, the integrand above is dominated by $\frac{1-\cos(z)}{z^2}$, which is integrable. It follows from the Dominated Convergene Theorem that 
\begin{equation}\label{eqn:cv-gle7-18}
\lim_{t\to\infty}I_5(t)=\rhat\left(0\right)\int_0^{\infty}  \frac{1-\cos(z)}{z^2}
\d z\in(0,\infty).
\end{equation}
We finally combine \eqref{eqn:cv-gle7-16} and \eqref{eqn:cv-gle7-18} to obtain \eqref{eqn:cv-gle7-14}, which concludes part (a).

(b) Firstly, in view of Proposition \ref{prop:fourier.taube}, since $K(t)$ satisfies \eqref{cond2b}, we have that
\begin{equation}\label{lim:cv-gle7-19a}
\lim_{\omega\to 0}\omega^{1-\alpha}\rhat(\omega)=\frac{2}{\pi\beta}\lim_{\omega\to 0}\omega^{1-\alpha}\frac{\Kcos(\omega)}{\Kcos(\omega)^2+\Ksin(\omega)^2}=C\in (0,\infty).
\end{equation}
Since
\begin{equation}\label{eqn:cv-gle7-19}
\frac{\E{X^2(t)}}{t^\alpha}=2t^{1-\alpha}\int_\rbb \frac{1-\cos(z)}{z^2}
\rhat\left(\frac{z}{t}\right)\d z,
\end{equation}
it suffices to show that
\begin{equation}\label{eqn:cv-gle7-20}
\lim_{t\to\infty}t^{1-\alpha}\int_0^\infty \frac{1-\cos(z)}{z^2}
\rhat\left(\frac{z}{t}\right)\d z=C\in(0,\infty).
\end{equation}
Fixing the same $A$ from part (a), we have
\begin{equation}\label{eqn:cv-gle7-21}
\int_0^\infty t^{1-\alpha} \frac{1-\cos(z)}{z^2}
\rhat\left(\frac{z}{t}\right)\d z=\int_0^{At}+\int_{At}^\infty t^{1-\alpha}\frac{1-\cos(z)}{z^2}
\rhat\left(\frac{z}{t}\right)\d z=I_7(t)+I_8(t).
\end{equation}
To see that $I_8(t)\to 0$ as $t\to\infty$, we have a chain of implications
\begin{equation}\label{eqn:cv-gle7-22}
	\begin{aligned}
		I_8(t) \leq t^{1-\alpha}\int_{At}^\infty c \, \frac{1-\cos(z)}{z^2} \left(\frac{z}{t} \right)^{\sigma} \d z &= \frac{ct^{1-\alpha}}{t^\sigma}\int_{At}^\infty \frac{1-\cos(z)}{z^{2-\sigma}}\d z\\
			&\leq \frac{ct^{1-\alpha}}{t^\sigma}\times\frac{1}{A^{1-\sigma}t^{1-\sigma}}=\frac{c}{t^\alpha}\to 0,
	\end{aligned}	
\end{equation}
where the constant $c$ may change from line to line independent of $t$. Next, we write $I_7(t)$ as
\begin{equation}\label{eqn:cv-gle7-23}
I_7(t)=\int_0^{\infty} 1_{(0,At]}(z) \frac{1-\cos(z)}{z^{1+\alpha}}
\left(\frac{z}{t}\right)^{1-\alpha}\rhat\left(\frac{z}{t}\right)\d z.
\end{equation}
From \eqref{lim:cv-gle7-19a}, we observe that $\omega^{1-\alpha}\rhat(\omega)$ is bounded on $(0,A]$. Since $z/t\leq A$, the integrand above is dominated by $\frac{1-\cos(z)}{z^{1+\alpha}}$, which is integrable. Taking $t$ to infinity, it follows from the Dominated Convergene Theorem that 
\begin{equation}\label{eqn:cv-gle7-24}
\lim_{t\to\infty}I_7(t)=\int_0^{\infty}  \frac{1-\cos(z)}{z^{1+\alpha}}
\d z\lim_{\omega\to 0}\omega^{1-\alpha}\rhat\left(\omega\right)\in(0,\infty).
\end{equation}
Finally, \eqref{eqn:cv-gle7-20} follows from \eqref{eqn:cv-gle7-22} and \eqref{eqn:cv-gle7-24}. The proof is thus complete.
\end{proof}

\section{Transient Anomalous Diffusion}
\label{sec:transient-anomalous-diffusion} 

As mentioned when we introduced the Generalized Rouse family of memory kernels in Section \ref{sec:prelim:exponentials}, a sum of exponentials can be used to approximate a power law. This is an appealing property because, for such memory kernels, the non-Markov GLE can be rewritten as a high-dimensional system of SDEs. (See \cite{goychuk2012viscoelastic} or \cite{pavliotis2014stochastic} for discussion of the finite-dimensional case.) Since a finite sum of exponentials will always be integrable, the associated solutions to the GLE will be asymptotically diffusive.  Nevertheless, the MSD of these solutions will look subdiffusive over a large time range if the memory kernel has an appropriate form. 

In this section, we propose a rigorous definition of \emph{transient anomalous diffusion} in the case where either $m>0$ or $\lambda>0$. We formulate the result in such a way that one can check a convergence condition on the sequence of memory kernels and then have that for any interval $[0,T]$, there is an $N$ sufficiently large so that the GLE with $N$ terms is arbitrarily close to the limiting MSD over $[0,T]$. One might think that such a result is automatic, but the argument is more subtle than expected. We provide some results in this direction. Once again, the analysis is more subtle when $m > 0$ and $\lambda = 0$ (Theorem \ref{thm4}) and easier when $\lambda > 0$ (Theorem \ref{thm:transient:lambdapos}). However, we do not have a result of this kind for $m = \lambda = 0$.

\begin{theorem}\label{thm4} Suppose that $m>0$ and $\lambda=0$. Assume all of the following.
\begin{enumerate}[(a)]
\item $K_n\in C^2(0,\infty)$ satisfies \eqref{cond1} $+$ \eqref{cond2a}.
\item  $K\in  C^2(0,\infty)$ satisfies \eqref{cond1} $+$ \eqref{cond2b}. 
\item For all $n\in\nbb$, $K_n(t)$ is convex on $(0,\infty)$.
\item As $n\rightarrow\infty$, $K_n(t)\longrightarrow K(t)$ for all $t>0$.
\item There exists a constant $\kappa\in(0,1)$ such that 
\begin{equation}\label{cond:uniform-bound}
\sup_{n\in\nbb^+} \sup_{t\in(0,1]} t^\kappa K_n(t)<\infty.
\end{equation}

\end{enumerate}
Let $X_n$, $X$ be the particle position processes as in Corollary \ref{cor:X(t)-diff} associated with $K_n$, $K$, respectively. Then for all $T>0$,
\begin{equation}\label{eqn2}
\lim_{n\to\infty}\bigg[\sup_{t,s\in[0,T]} \left|\E{X_n(t)X_n(s)}-\E{X(t)X(s)}\right| \bigg]=0.
\end{equation}
\end{theorem}

In order to prove Theorem \ref{thm4}, we need some preliminary facts.
\begin{lemma} \label{lem:trig-ineq}
For $x,y\in\rbb$, there holds
\begin{equation*} 
\left| \cos(x-y)-\cos(x)-\cos(y)+1 \right| \leq 2-\cos(x)-\cos(y).
\end{equation*}
\end{lemma}
\begin{proof}
Our inequality is equivalent to
\begin{displaymath}
-2+\cos(x)+\cos(y)\leq\cos(x-y)-\cos(x)-\cos(y)+1  \leq 2-\cos(x)-\cos(y).
\end{displaymath}
The right hand side inequality is evident. We are left to prove
\begin{displaymath}
-2+\cos(x)+\cos(y)\leq\cos(x-y)-\cos(x)-\cos(y)+1,
\end{displaymath}
which can be written as
\begin{multline*}2\left[\sin^2(x/2)+\sin^2(y/2)+2\cos(x/2)\cos(y/2)\sin(x/2)\sin(y/2)\right]\\+\left(1-\cos(x)\right)\left(1-\cos(y)\right)\geq 0,
\end{multline*}
which in turn always holds.
\end{proof}

We now assert that the Fourier cosine and sine transforms of $K_n$ converge pointwise to those of $K$.

\begin{proposition} \label{prop:fourier.limit}
Suppose that $\{K_n\}_{n\geq 1}$ and $K$ satisfy \eqref{cond1} and that for every $t>0$, $K_n(t)\rightarrow K(t)$ as $n\to\infty$. For each $n\geq 1$, put 
\[\Kcos^n(\omega)=\int_0^\infty K_n(t)\cos(t\omega)\d t,\qquad \Ksin^n(\omega)=\int_0^\infty K_n(t)\sin(t\omega)\d t.\]
Then, for non-zero $\omega$,
\begin{equation} \label{prop.fourier.limit.1}
\lim_{k\to\infty}\Kcos^n(\omega)=\Kcos(\omega) \text{ and }  \lim_{k\to\infty}\Ksin^n(\omega)=\Ksin(\omega).
\end{equation}
\end{proposition}
\begin{proof}
Given $\varepsilon>0$, fix $A$ large enough such that 
\begin{equation} \label{prop.fourier.limit.1c}
\left| \int_0^\infty K(t)\cos(t\omega)\d t-\int_0^A K(t)\cos(t\omega)\d t\right|<\varepsilon,
\quad\text{and}\quad \frac{8K(A)}{\omega}<\varepsilon,
\end{equation}
where the latter condition is possible since $K(t)$ eventually decreases to 0 as $t\to\infty$. We have
\begin{equation}\label{prop.fourier.limit.1d}
\int_0^A K_n(t)\cos(t\omega)\d t\overset{n\to\infty}{\longrightarrow}\int_0^A f(t)\cos(t\omega)\d t,
\end{equation}
by virtue of the Dominated Convergence Theorem. For each $n$, we note that Inequality \eqref{ineq:fourier.1a} also holds for $K_n$, which implies
\begin{equation}\label{prop.fourier.limit.1e}
\left|\int_A^\infty K_n(t)\cos(t\omega)\d t\right| \leq \frac{4K_n(A)}{\omega}\leq \frac{8K(A)}{\omega}<\varepsilon,
\end{equation}
since $K_n(A)\rightarrow K(A)$ as $n\to\infty$. It follows immediately that
\begin{equation*}
\left| \int_0^\infty K_n(t)\cos(t\omega)\d t-\int_0^A K_n(t)\cos(t\omega)\d t\right|<\varepsilon,
\end{equation*}
which is equivalent to
\begin{equation}\label{prop.fourier.limit.1f}
\int_0^A K_n(t)\cos(t\omega)\d t-\varepsilon < \int_0^\infty K_n(t)\cos(t\omega)\d t <\int_0^A K_n(t)\cos(t\omega)\d t+\varepsilon.
\end{equation}
We now send $n$ to infinity and combine \eqref{prop.fourier.limit.1c}, \eqref{prop.fourier.limit.1d} and \eqref{prop.fourier.limit.1f} to obtain the Fourier cosine limit in \eqref{prop.fourier.limit.1}. A similar argument is applied to establish the Fourier sine limit. 
\end{proof}

As a direct consequence of Proposition \ref{prop:fourier.limit}, we obtain uniform bounds on $\{\Kcos^n\}_{n\geq 1}$ and $\{\Ksin^n\}_{n\geq 1}$ in the following lemma.
\begin{lemma} \label{prop:r.hat} Let $K_n$, $K$ be as in Theorem \ref{thm4}. Then for every $\omega_0>1$, there exists $N>0$ sufficiently large such that
\begin{equation} \label{ineq:Kcos-bound}
\inf_{n\geq N} \inf_{\omega\in(0,\omega_0]} \Kcos^n(\omega) > 0, 
\end{equation}
and
\begin{equation} \label{ineq:Ksin-bound}
\sup_{n\geq N} \sup_{\omega>\omega_0} \Kcos^n(\omega) <\infty, \,\, \sup_{n\geq N} \sup_{\omega>\omega_0} \Ksin^n(\omega) <\infty. 
\end{equation}
\begin{equation} \label{r.hat.limit}
\lim_{n\to\infty}\int_0^\infty \left|\widehat{r}_n(\omega)-\widehat{r}(\omega)\right|\d\omega=0.
\end{equation}
\end{lemma}
\begin{proof}
We first note that $\left\{K_n\right\}_{n\geq 1}$ are convex, and so is $K$ being the limiting function. Furthermore, convexity and eventually decreasing to zero imply that $K_n(t)$ is actually decreasing to zero for $t\in [0,\infty)$. 

We now invoke \eqref{cond:uniform-bound} to see that $\lim_{t\to 0} tK_n(t)=0$. In view of Lemma~\ref{lem4}, $\Kcos^n$ satisfies formula \eqref{eqn:lem4.1}. We then estimate
\begin{multline*}
\int_0^\infty K_n(t)\cos(t\omega)\d t=\frac{1}{\omega^2}\int_0^\infty K_n''(t)(1-\cos(t\omega))\d t \\
\geq  \int_0^{t_1} K_n''(t)\frac{1-\cos(t\omega)}{\omega^2}\d t
= \int_0^{t_1} t^2 K_n''(t)\frac{1-\cos(t\omega)}{(t\omega)^2}\d t.
\end{multline*}
Observe that $\min_{x\in (0,\pi/2]}\frac{1-\cos(x)}{x^2}=c_1>0$. Fixing $\omega_0>1$ and setting $t_1=\pi/(2\omega_0)$, for $\omega \in (0,\omega_0]$, we have
\begin{equation} \label{r.hat.0}
\frac{1}{\omega^2}\int_0^\infty K_n''(t)(1-\cos(t\omega))\d t \geq c_1 \int_0^{t_1} t^2 K_n''(t)\d t.
\end{equation}
Integrating by parts the above RHS yields
\begin{equation*}
\int_0^{t_1} t^2 K_n''(t)\d t = t_1^2 K_n'(t_1)-t_1 K_n(t_1)+\int_0^{t_1} K_n(t)\d t.
\end{equation*}
Fix $0<t^*<t_1$ to be chosen later. The Mean Value Theorem implies
\begin{eqnarray*}
\frac{K_n(t^*)-K_n(t_1)}{t^*-t_1}&=&K_n'(\xi),\qquad t^* < \xi< t_1,\\
&\leq & K_n'(t_1),
\end{eqnarray*}
since $K_n'(t)$ is increasing on $t\in(0,\infty)$. It follows that
\begin{equation*}
\int_0^{t_1} t^2 K_n''(t)\d t \geq  t_1^2\frac{K_n(t^*)-K_n(t_1)}{t^*-t_1}-t_1 K_n(t_1)+\int_0^{t_1} K_n(t)\d t.
\end{equation*}
Letting $n\rightarrow\infty$, we obtain
\begin{align*}
\MoveEqLeft[4]\liminf_{n\to\infty} \int_0^{t_1} t^2 K_n''(t)\d t\\
 &\geq t_1^2\frac{K(t^*)-K(t_1)}{t^*-t_1}-t_1 K(t_1)+\int_0^{t_1} K(t)\d t\\
&=  t_1^2\left[\frac{K(t^*)-K(t_1)}{t^*-t_1}-K'(t_1)\right]+t_1^2K'(t_1)-t_1 K(t_1)+\int_0^{t_1} K(t) \d t\\
&= t_1^2\left[\frac{K(t^*)-K(t_1)}{t^*-t_1}-K'(t_1)\right] + \int_0^{t_1} t^2 K''(t) \d t.
\end{align*}
As $t^*\rightarrow t_1$, on the RHS above, the bracket tends to 0 whereas the integral is positive. Subsequently, we can choose $t^*$ close enough to $t_1$ such that the RHS above is positive. And thus,
\begin{equation} \label{r.hat.4}
\liminf_{n\to\infty} \int_0^{t_1} t^2 K_n''(t) \d t > 0.
\end{equation}
We then combine \eqref{r.hat.4} with \eqref{r.hat.0} to infer the existence of a constant $c_2=c_2(\omega_0)>0$ such that for $n$ large and $\omega \in (0,\omega_0]$, it holds that
\begin{equation} \label{r.hat.1}
\int_0^\infty K_n(t)\cos(t\omega) \d t \geq c_2>0,
\end{equation}
which proves \eqref{ineq:Kcos-bound}. Now for $\omega>\omega_0$, by a change of variable, we have
\begin{eqnarray*}
\int_0^\infty K_n(t)\cos(t\omega) \d t &=& \frac{1}{\omega}\int_0^\infty K_n\left(\frac{z}{\omega}\right)\cos(z) \d z\\
&=& \frac{1}{\omega^{1-\kappa}} \int_0^1 \left(\frac{z}{\omega}\right)^\kappa K_n\left(\frac{z}{\omega}\right)\frac{\cos(z)}{z^\kappa} \d z +\frac{1}{\omega} \int_1^\infty K_n\left(\frac{z}{\omega}\right)\cos(z) \d z
\end{eqnarray*}
To estimate the first integral on the above RHS, we employ the fact that $t^\kappa K_n(t)$ is uniformly bounded on $(0,1]$ from \eqref{cond:uniform-bound} to find
\begin{align*}
\frac{1}{\omega^{1-\kappa}} \int_0^1 \left(\frac{z}{\omega}\right)^\kappa K_n\left(\frac{z}{\omega}\right)\frac{\cos(z)}{z^\kappa} \d z &\leq \frac{\sup_{n\in\nbb^+} \sup_{t\in(0,1]} t^\kappa K_n(t)}{\omega^{1-\kappa}}\int_0^1 \frac{\cos(z)}{z^\kappa} \d z\\
&\leq \sup_{n\in\nbb^+} \sup_{t\in(0,1]} t^\kappa K_n(t)\int_0^1 \frac{\cos(z)}{z^\kappa} \d z,
\end{align*} 
where the last implication simply follows from the assumption $\omega>\omega_0>1$. For the other integral, invoking the Second Mean Value Theorem again gives
\begin{align*}
\frac{1}{\omega} \int_1^\infty K_n\left(\frac{z}{\omega}\right)\cos(z) \d z\leq
\frac{2}{\omega}K_n\left(\frac{1}{\omega}\right)\leq 2\sup_{n\in\nbb^+} \sup_{t\in(0,1]} t K_n(t)\leq 2\sup_{n\in\nbb^+} \sup_{t\in(0,1]} t^\kappa K_n(t),
\end{align*}
since $\kappa\in(0,1)$ by Assumption \emph{(d)} of Theorem~\ref{thm4}. We thus obtain for every $\omega\geq\omega_0$ and $n>0$
\begin{equation} \label{r.hat.2}
\int_0^\infty K_n(t)\cos(t\omega) \d t  \leq \bigg(2+\int_0^1 \frac{\cos(z)}{z^{\kappa}}\d z\bigg)\sup_{n\in\nbb^+} \sup_{t\in(0,1]} t^\kappa K_n(t),
\end{equation}
and likewise, 
\begin{equation} \label{r.hat.3}
\int_0^\infty K_n(t)\sin(t\omega) \d t  \leq\bigg(2+\int_0^1 \frac{\cos(z)}{z^{\kappa}}\d z\bigg)\sup_{n\in\nbb^+} \sup_{t\in(0,1]} t^\kappa K_n(t),
\end{equation}
which proves \eqref{ineq:Ksin-bound}. The proof is thus complete.
\end{proof}

We are now ready to give 
\begin{proof}[Proof of Theorem \ref{thm4}]
A short computation yields 
\begin{equation}\label{eqn3}
\E{X_n(t)X_n(s)}=\int_\rbb \frac{\cos((t-s)\omega)-\cos(t\omega)-\cos(s\omega)+1}{\omega^2}\ 
\widehat{r}_n(\omega)\d\omega.
\end{equation}
For $t,s\in[0,T]$, we estimate
\begin{align*}
\MoveEqLeft[4] \left|\E{X_n(t)X_n(s)}-\E{X(t)X(s)}\right| \\
&\leq \int_\rbb \frac{\left|\cos((t-s)\omega)-\cos(t\omega)-\cos(s\omega)+1\right|}{\omega^2}\left|\widehat{r_n}(\omega)-\widehat{r}(\omega)\right| \d\omega\\
&\leq \int_\rbb \frac{2-\cos(t\omega)-\cos(s\omega)}{\omega^2}\left|\widehat{r_n}(\omega)-\widehat{r}(\omega)\right| \d\omega \\
&=\int_\rbb \frac{1-\cos(t\omega)}{\omega^2}\lvert\widehat{r_n}(\omega)-\widehat{r}(\omega)\rvert \d\omega
+ \int_\rbb \frac{1-\cos(s\omega)}{\omega^2}\lvert\widehat{r_n}(\omega)-\widehat{r}(\omega)\rvert \d\omega.
\end{align*}
It follows that
\begin{align} \label{eqn4}
\sup_{t,s\in[0,T]} \left|\E{X_n(t)X_n(s)}-\E{X(t)X(s)}\right|
\leq 2\sup_{t\in[0,T]} \int_\rbb \frac{1-\cos(t\omega)}{\omega^2}\lvert\widehat{r_n}(\omega)-\widehat{r}(\omega)\rvert \d\omega.
\end{align}
We note that for $0\leq t\leq T$,
\[
\frac{1-\cos(t\omega)}{\omega^2}=t^2\frac{1-\cos(t\omega)}{(t\omega)^2}\leq T^2 \sup_{\omega\in\rbb}\frac{1-\cos(\omega)}{\omega^2}.
\]
We then combine with Inequality \eqref{eqn4} to see that
\begin{multline} \label{eqn5}
\sup_{t,s\in[0,T]} \left|\E{X_n(t)X_n(s)}-\E{X(t)X(s)}\right|
\\ \leq 2T^2\int_\rbb\lvert\widehat{r}_n(\omega)-\widehat{r}(\omega)\rvert \d\omega \sup_{\omega\in\rbb\setminus\{0\}}\frac{1-\cos(\omega)}{\omega^2} .
\end{multline}
The problem now is reduced to showing that $\widehat{r}_n\rightarrow\widehat{r}$ in $L^1(\rbb)$. In view of Proposition~\ref{prop:fourier.limit}, for $\omega > 0$, $\widehat{r}_n(\omega)=\widehat{r}(\omega)$ as $n\to\infty$. It remains to find a dominating function. Let $\omega_0$ be the constant from Lemma~\ref{prop:r.hat}. There are two cases: on one hand, if $\omega \leq\omega_0$, recalling Formula~\eqref{r.hat}, we have
\begin{eqnarray*}
\pi\widehat{r}_n(\omega) &=& \frac{2\int_0^\infty K_n(t)\cos(t\omega) \d t}
{\left[\int_0^\infty K_n(t)\cos(t\omega) \d t \right]^2
+\left[m\omega- \int_0^\infty K_n(t)\sin(t\omega) \d t \right]^2}\\
&\leq & \frac{2}{\int_0^\infty K_n(t)\cos(t\omega) \d t}\leq  \frac{2}{c_2},
\end{eqnarray*}
where $c_2=c_2(\omega_0)$ is the constant in \eqref{r.hat.1}. On the other hand, observing that the constant $c_3:=\bigg(2+\int_0^1 \frac{\cos(z)}{z^{\kappa}}\d z\bigg)\sup_{n\in\nbb^+} \sup_{t\in(0,1]} t^\kappa K_n(t)$ in \eqref{r.hat.2} and \eqref{r.hat.3} does not depend on $\omega_0$. We thus can choose $\omega_0$ large enough such that $c_3/\omega_0<m$. Hence, for $\omega>\omega_0$, we have
\begin{equation*}
\pi\widehat{r}_n(\omega) \leq \frac{2\int_0^\infty K_n(t)\cos(t\omega) \d t}
{\left[m\omega- \int_0^\infty K_n(t)\sin(t\omega) \d t \right]^2}\leq  \frac{2c_3}{\omega^2\left[m-c_3/\omega_0\right]^2},
\end{equation*}
where the last implication follows from \eqref{r.hat.2} and \eqref{r.hat.3}.
Combining two cases above, we infer the following function 
\begin{displaymath}
g(\omega)=\frac{2}{c_2} 1_{(0,\omega_0]}(\omega)+\frac{2c_3}{\omega^2\left[m-c_3/\omega_0\right]^2} 1_{(\omega_0,\infty)}(\omega),
\end{displaymath}
dominating $\widehat{r}_n(\omega)$ in $\rbb$. It is also clear that $g\in L^1(\rbb)$. The Dominated Convergence Theorem then implies that $\widehat{r_n}$ converges to $\widehat{r}$ in $L^1(\rbb)$. As a consequence, we obtain \eqref{eqn2} following from \eqref{eqn5}. The proof is thus complete.
\end{proof}

We finally assert a result similar to Theorem \ref{thm4} for the case $m\geq 0,\,\lambda>0$, in which minimal assumptions on memory kernels are required. 
\begin{theorem}\label{thm:transient:lambdapos} Suppose that $m\geq 0$ and $\lambda>0$. Assume all of the following.
\begin{enumerate}[(a)]
\item $K_n$ satisfies \eqref{cond1} $+$ \eqref{cond2a}.
\item  $K$ satisfies \eqref{cond1} $+$ \eqref{cond2b}. 
\item As $n\rightarrow\infty$, $K_n(t)\longrightarrow K(t)$ for all $t>0$.

\end{enumerate}
Let $X_n(t)=\la V_n,1_{[0,t]}\ra$, $X(t)=\la V,1_{[0,t]}\ra$ be the particle position processes as in Definition~\ref{defn:form:V(t)} where $V_n,\, V$ are as in either Theorem~\ref{thm:weak-solution-mpos-lambdapos} (when $m>0,\, \lambda>0$) or Theorem~\ref{thm:weak-solution-mzero-lambdapos} (when $m=0,\, \lambda>0$) associated with $K_n$, $K$, respectively. Then for all $T>0$,
\begin{equation}\label{eqn:transient:lambdapos}
\lim_{n\to\infty}\bigg[\sup_{t,s\in[0,T]} \big|\E{X_n(t)X_n(s)}-\E{X(t)X(s)}\big| \bigg]=0.
\end{equation}
\end{theorem}
\begin{proof} Let $\widehat{r}_n,\,\widehat{r}$ be spectral densities associated with $X_n$ and $X$, respectively. Note that Inequality~\eqref{eqn4} is still valid regardless of $\lambda$,
\begin{align} \label{ineq:transient:lambdapos:1}
\sup_{t,s\in[0,T]} \left|\E{X_n(t)X_n(s)}-\E{X(t)X(s)}\right|
\leq 2\sup_{t\in[0,T]} \int_\rbb \frac{1-\cos(t\omega)}{\omega^2}\lvert\widehat{r_n}(\omega)-\widehat{r}(\omega)\rvert \d\omega.
\end{align}
Now, for $0\leq t\leq T$, we have
\begin{equation}\label{ineq:transient:lambdapos:2}
\begin{aligned}
\MoveEqLeft[3]\int_\rbb \frac{1-\cos(t\omega)}{\omega^2}\lvert\widehat{r}_n(\omega)-\widehat{r}(\omega)\rvert \d\omega\\
 &= \int_{|\omega|\leq 1} \frac{1-\cos(t\omega)}{\omega^2}\lvert\widehat{r}_n(\omega)-\widehat{r}(\omega)\rvert \d\omega+\int_{|\omega|>1} \frac{1-\cos(t\omega)}{\omega^2}\lvert\widehat{r}_n(\omega)-\widehat{r}(\omega)\rvert \d\omega\\
 &\leq T^2\sup_{y\in\rbb}\frac{1-\cos(y)}{y^2}\int_{|\omega|\leq 1} \left|\widehat{r}_n(\omega)-\widehat{r}(\omega)\right| \d\omega+\int_{|\omega|>1} \frac{2}{\omega^2}\lvert\widehat{r}_n(\omega)-\widehat{r}(\omega)\rvert \d\omega\\
\end{aligned}
\end{equation}
Recalling from~\eqref{form:rhat-mpos-lambdapos} that $\widehat{r_n}$ is given by
\begin{align*}
\pi\widehat{r}_n(\omega) &= \frac{\lambda+\beta\int_0^\infty K_n(t)\cos(t\omega) \d t}
{\left[\lambda+\beta\int_0^\infty K_n(t)\cos(t\omega) \d t \right]^2
+\left[m\omega- \beta\int_0^\infty K_n(t)\sin(t\omega) \d t \right]^2},
\end{align*}
which immediately yields the estimate $\widehat{r}_n(\omega)\leq \frac{1}{\pi\lambda}$, and likewise, $\widehat{r}(\omega)\leq \frac{1}{\pi\lambda}$. In addition, since $K_n$ converges to $K$ pointwise, Proposition~\ref{prop:fourier.limit} implies that $\widehat{r_n}(\omega)$ converges to $\widehat{r}(\omega)$ for every non-zero $\omega$. It follows from the Dominated Convergence Theorem that
\begin{equation} \label{eqn:transient:lambdapos:3}
\lim_{n\to\infty}\int_{|\omega|\leq 1} \left|\widehat{r}_n(\omega)-\widehat{r}(\omega)\right| \d\omega=0,\quad\text{and}\quad\lim_{n\to\infty}\int_{|\omega|>1} \frac{2}{\omega^2}\lvert\widehat{r}_n(\omega)-\widehat{r}(\omega)\rvert \d\omega=0.
\end{equation}
We finally combine \eqref{ineq:transient:lambdapos:1}, \eqref{ineq:transient:lambdapos:2} and \eqref{eqn:transient:lambdapos:3} to deduce the desired limit~\eqref{eqn:transient:lambdapos}, which concludes the proof.
\end{proof}
\section*{Acknowledgements}
The authors would like to thank Greg Forest, Christel Hohenegger, Pete Kramer, Martin Lysy and Davar Khoshnevisan for helpful conversations in the development of this work.

\bibliographystyle{siamplain}
\bibliography{gle-anomalous-diffusion}

\begin{thebibliography}{10}

\bibitem{cramer1967stationary}
{\sc H.~Cram{\'e}r and M.~R. Leadbetter}, {\em {Stationary and related
  stochastic processes: Sample function properties and their applications}},
  Courier Corporation, 1967.

\bibitem{didier2012statistical}
{\sc G.~Didier, S.~A. McKinley, D.~B. Hill, and J.~Fricks}, {\em Statistical
  challenges in microrheology}, Journal of Time Series Analysis, 33 (2012),
  pp.~724--743.

\bibitem{ernst2012fractional}
{\sc D.~Ernst, M.~Hellmann, J.~K{\"o}hler, and M.~Weiss}, {\em Fractional
  brownian motion in crowded fluids}, Soft Matter, 8 (2012), pp.~4886--4889.

\bibitem{feller1966introduction}
{\sc W.~Feller}, {\em An introduction to probability theory and its
  applications}, vol.~2, John Wiley \& Sons, 1966.

\bibitem{fricks2009time}
{\sc J.~Fricks, L.~Yao, T.~C. Elston, and M.~G. Forest}, {\em Time-domain
  methods for diffusive transport in soft matter}, SIAM journal on applied
  mathematics, 69 (2009), pp.~1277--1308.

\bibitem{golding2006physical}
{\sc I.~Golding and E.~C. Cox}, {\em Physical nature of bacterial cytoplasm},
  Physical review letters, 96 (2006), p.~098102.

\bibitem{goychuk2012viscoelastic}
{\sc I.~Goychuk}, {\em {Viscoelastic subdiffusion: Generalized Langevin
  equation approach}}, Advances in Chemical Physics, 150 (2012), p.~187.

\bibitem{hairer2009introduction}
{\sc M.~Hairer}, {\em {An introduction to stochastic PDEs}}, arXiv preprint
  arXiv:0907.4178,  (2009).

\bibitem{hall2016uncertainty}
{\sc E.~J. Hall, M.~A. Katsoulakis, and L.~Rey-Bellet}, {\em Uncertainty
  quantification for generalized langevin dynamics}, The Journal of chemical
  physics, 145 (2016), p.~224108.

\bibitem{hill2014biophysical}
{\sc D.~B. Hill, P.~A. Vasquez, J.~Mellnik, S.~A. McKinley, A.~Vose, F.~Mu,
  A.~G. Henderson, S.~H. Donaldson, N.~E. Alexis, R.~C. Boucher, et~al.}, {\em
  A biophysical basis for mucus solids concentration as a candidate biomarker
  for airways disease}, {PloS one}, 9 (2014), p.~e87681.

\bibitem{hohenegger2017equipartition}
{\sc C.~Hohenegger}, {\em {On equipartition of energy and integrals of
  Generalized Langevin Equations with generalized Rouse kernel}},
  Communications in Mathematical Sciences, 15 (2017), pp.~539--554.

\bibitem{hohenegger2017reconstructing}
{\sc C.~Hohenegger and S.~A. McKinley}, {\em {Reconstructing complex fluid
  properties from the behavior of fluctuating immersed particles}}, Submitted
  and in review.

\bibitem{hohenegger2017fluid}
{\sc C.~Hohenegger and S.~A. McKinley}, {\em Fluid--particle dynamics for
  passive tracers advected by a thermally fluctuating viscoelastic medium},
  Journal of Computational Physics, 340 (2017), pp.~688--711.

\bibitem{indei2012treating}
{\sc T.~Indei, J.~D. Schieber, A.~C{\'o}rdoba, and E.~Pilyugina}, {\em Treating
  inertia in passive microbead rheology}, Physical Review E, 85 (2012),
  p.~021504.

\bibitem{inoue1995abel}
{\sc A.~Inoue}, {\em {On Abel-Tauber theorems for Fourier cosine transforms}},
  Journal of mathematical analysis and applications, 196 (1995), pp.~764--776.

\bibitem{ito1954stationary}
{\sc K.~It{\^o}}, {\em Stationary random distributions}, Memoirs of the College
  of Science, University of Kyoto. Series A: Mathematics, 28 (1954),
  pp.~209--223.

\bibitem{kneller2011generalized}
{\sc G.~R. Kneller}, {\em {Generalized Kubo relations and conditions for
  anomalous diffusion: physical insights from a mathematical theorem}}, The
  Journal of chemical physics, 134 (2011), p.~224106.

\bibitem{kou2004generalized}
{\sc S.~Kou and X.~S. Xie}, {\em {Generalized Langevin equation with fractional
  Gaussian noise: subdiffusion within a single protein molecule}}, Physical
  review letters, 93 (2004), p.~180603.

\bibitem{kou2008stochastic}
{\sc S.~C. Kou}, {\em Stochastic modeling in nanoscale biophysics: subdiffusion
  within proteins}, The Annals of Applied Statistics,  (2008), pp.~501--535.

\bibitem{kubo1966fluctuation}
{\sc R.~Kubo}, {\em The fluctuation-dissipation theorem}, Reports on progress
  in physics, 29 (1966), p.~255.

\bibitem{kupferman2004fractional}
{\sc R.~Kupferman}, {\em {Fractional kinetics in Kac--Zwanzig heat bath
  models}}, Journal of statistical physics, 114 (2004), pp.~291--326.

\bibitem{lysy2016model}
{\sc M.~Lysy, N.~S. Pillai, D.~B. Hill, M.~G. Forest, J.~W. Mellnik, P.~A.
  Vasquez, and S.~A. McKinley}, {\em Model comparison and assessment for single
  particle tracking in biological fluids}, Journal of the American Statistical
  Association, 111 (2016), pp.~1413--1426.

\bibitem{mason2000estimating}
{\sc T.~G. Mason}, {\em {Estimating the viscoelastic moduli of complex fluids
  using the generalized Stokes--Einstein equation}}, Rheologica Acta, 39
  (2000), pp.~371--378.

\bibitem{mason1995optical}
{\sc T.~G. Mason and D.~Weitz}, {\em Optical measurements of
  frequency-dependent linear viscoelastic moduli of complex fluids}, Physical
  review letters, 74 (1995), p.~1250.

\bibitem{mckinley2009transient}
{\sc S.~A. McKinley, L.~Yao, and M.~G. Forest}, {\em Transient anomalous
  diffusion of tracer particles in soft matter}, Journal of Rheology, 53
  (2009), pp.~1487--1506.

\bibitem{morgado2002relation}
{\sc R.~Morgado, F.~A. Oliveira, G.~G. Batrouni, and A.~Hansen}, {\em Relation
  between anomalous and normal diffusion in systems with memory}, Physical
  review letters, 89 (2002), p.~100601.

\bibitem{mori1965transport}
{\sc H.~Mori}, {\em {Transport, Collective Motion, and Brownian Motion*}},
  Progress of theoretical physics, 33 (1965), pp.~423--455.

\bibitem{pavliotis2014stochastic}
{\sc G.~A. Pavliotis}, {\em Stochastic processes and applications}, Springer,
  2014.

\bibitem{reif1965fundamentals}
{\sc F.~Reif}, {\em Fundamentals of statistical and thermal physics}, Waveland
  Press, 1965.

\bibitem{rudin1964principles}
{\sc W.~Rudin et~al.}, {\em Principles of mathematical analysis}, vol.~3,
  McGraw-hill New York, 1964.

\bibitem{schilling2012bernstein}
{\sc R.~L. Schilling, R.~Song, and Z.~Vondracek}, {\em Bernstein functions:
  theory and applications}, vol.~37, Walter de Gruyter, 2012.

\bibitem{soni1975slowly}
{\sc K.~Soni and R.~Soni}, {\em {Slowly varying functions and asymptotic
  behavior of a class of integral transforms I}}, Journal of Mathematical
  Analysis and Applications, 49 (1975), pp.~166--179.

\bibitem{squires2010fluid}
{\sc T.~M. Squires and T.~G. Mason}, {\em Fluid mechanics of microrheology},
  Annual review of fluid mechanics, 42 (2010).

\bibitem{tuck2006positivity}
{\sc E.~O. Tuck}, {\em {On positivity of Fourier transforms}}, Bulletin of the
  Australian Mathematical Society, 74 (2006), pp.~133--138.

\end{thebibliography}

\end{document}